\documentclass[12pt,reqno]{amsart}

\makeatletter
\@namedef{subjclassname@2020}{%
	\textup{2020} Mathematics Subject Classification}
\makeatother

\usepackage{amsthm,amsmath,amssymb,amscd}
\allowdisplaybreaks
\usepackage{mathrsfs}
\usepackage[colorlinks,citecolor=magenta,linkcolor=blue]{hyperref}
\usepackage{amssymb}
\usepackage{esint}
\usepackage{graphics}
\newtheorem{theorem}{Theorem}[section]
\newtheorem{corollary}[theorem]{Corollary}
\newtheorem{assumption}[theorem]{Assumption}
\newtheorem{lemma}[theorem]{Lemma}
\newtheorem{proposition}[theorem]{Proposition}
\newtheorem{definition}{Definition}[section]
\theoremstyle{remark}
\newtheorem{remark}[theorem]{Remark}
\newtheorem*{ack}{Acknowledgments}

\numberwithin{equation}{section}
\begin{document}
\title[Almost rigidity of the Talenti-type comparison theorem]{Almost rigidity of the Talenti-type comparison theorem on $\mathrm{RCD}(0,N)$ spaces}
\author[Wenjing Wu]{Wenjing Wu}%${}^*$}
\address{School of Mathematical Sciences, University of Science and Technology of China, Hefei 230026, P.R. China}
\email{\href{mailto:wuwenjing@ustc.edu.cn}{wuwenjing@ustc.edu.cn}}

\begin{abstract}
In this paper, we prove a Talenti-type comparison theorem for the $p$-Laplacian with Dirichlet boundary conditions on open subsets of a $\mathrm{RCD}(0,N)$ space with $N\in (1,\infty)$. We also obtain an almost rigidity result of the Talenti-type comparison theorem, whose proof relies on a compactness on varying spaces.
\end{abstract}

\date{\today}
%\thanks{}

\maketitle
\tableofcontents

\section{Introduction}
Let $\Omega\subset \mathbb{R}^n$ be a bounded open set and $f\in {L}^2(\Omega)$. Let $u$ be the solution to 
\begin{equation*}
	\begin{cases}
		-\Delta u=f, &\text{ in }\Omega,\\
		u=0, &\text{ on }\partial\Omega,
	\end{cases}
\end{equation*}
and let $v$ be the solution to 
\begin{equation*}
	\begin{cases}
		-\Delta v=f^\sharp, &\text{ in }\Omega^\sharp,\\
		v=0, &\text{ on }\partial\Omega^\sharp,
	\end{cases}
\end{equation*}
where $\Omega^\sharp$ denotes the unique ball centered at the origin satisfying $\mathcal{L}^n(\Omega^\sharp)=\mathcal{L}^n(\Omega)$ and $f^\sharp$ is the Schwarz symmetrization of $f$. In 1976, Talenti \cite{Ta76} proved that
$$
u^\sharp (x)\leq v(x),\quad \mathcal{L}^n\text{-a.e. }x\in \Omega^\sharp.
$$
Moreover, if $u^\sharp=v$ $\mathcal{L}^n$-a.e. in $\Omega^\sharp$, then $\Omega$ is itself was already a ball. Talenti's comparison results were generalized to Riemannian manifolds with non-negative Ricci curvature in \cite{CL23}.

\vspace{0.4cm}

The aim of the present work is to generalize such a comparison result to a curved, possibly non-smooth, setting and prove an almost rigidity.

The framework of the paper is the one of metric measure spaces with non-negative Ricci curvature and dimension bounded above by $N\in (1,\infty)$ in a synthetic sense, via optimal transport.

Let $(\mathrm{X},\mathrm{d},\mathfrak{m})$ be a $\mathrm{RCD}(0,N)$ space for some $N\in (1,\infty)$, and having Euclidean volume growth, i.e.
$$
\mathrm{AVR}_{\mathfrak{m}}:=\lim_{r\rightarrow \infty}\frac{\mathfrak{m}(B_r(x_0))}{\omega_N r^N}>0,\quad \text{ for some (and thus any) }x_0\in \mathrm{supp}[\mathfrak{m}],
$$
where $\omega_N=\frac{\pi^{N/2}}{\int_0^\infty t^{N/2}\mathrm{e}^{-t}\mathrm{d}t}$.

In order to state the main theorems, let us introduce some notations about the $1$-dimensional model spaces and the corresponding Schwarz symmetrization. For any $N\in [1,\infty)$, we define the one-dimensional model space $([0,\infty),\mathrm{d}_{eu},\mathfrak{m}_{N}:=N\omega_N t^{N-1}\mathcal{L}^1\llcorner_{[0,\infty)})$, where $\mathrm{d}_{eu}$ is the restriction to $[0,\infty)$ of the canonical Euclidean distance over the real line and $\mathcal{L}^1$ is the standard Lebesgue measure.

We now introduce the corresponding Schwarz symmetrization. To this aim, given an open set $\Omega\subset \mathrm{X}$ with $\mathfrak{m}(\Omega):=a\in (0,\infty)$ and a Borel function $u:\Omega\rightarrow \mathbb{R}$, we define its distribution function $\mu:[0,\infty)\rightarrow [0,\mathfrak{m}(\Omega)]$ by 
$$
\mu(t):=\mathfrak{m}(\{|u|>t\}).
$$
We will let $u^\sharp:[0,\infty)\rightarrow [0,\infty]$ be the generalized inverse of $\mu$, defined in the following way:
$$
u^\sharp(s)=\begin{cases}
	\mathrm{ess}\text{-}\mathrm{sup} |u|,&\text{ if }s=0,\\
	\inf\{t\geq 0:\mu(t)<s\},&\text{ if }s>0.
\end{cases}
$$
Consider $\Omega^\star:=[0,r_a)$ such that $\mathfrak{m}_{N}([0,r_a])=\mathfrak{m}(\Omega)=a$, the Schwarz symmetrization $u^\star:\Omega^\star\rightarrow [0,\infty]$ is defined by
$$
u^\star(x):=u^\sharp(\mathfrak{m}_{N}([0,x])).
$$

Let $p, q \in(1, \infty)$ be conjugate exponents and let $f \in L^q(\Omega,\mathfrak{m})$ be a non-zero function, we consider the Poisson problem
\begin{equation}\label{A}
	\begin{cases}
		-\mathscr{L}_p u=f, & \text { in } \Omega, \\ 
		u=0, & \text { on } \partial \Omega.\end{cases}
\end{equation}
A function $u \in W^{1, p}(\Omega,\mathfrak{m})$ is a solution to $(\ref{A})$ if $u \in W_0^{1, p}(\Omega,\mathfrak{m})$ and
$$
\int_{\Omega} \langle \nabla u,\nabla \psi\rangle |\nabla u|^{p-2} \mathrm{d}\mathfrak{m}=\int_{\Omega} f \psi \mathrm{d}\mathfrak{m}, \quad \text{ for every } \psi \in W_0^{1, p}(\Omega,\mathfrak{m}),
$$

We will establish a comparison between the Schwarz symmetrization of the solution to $(\ref{A})$ and the solution to the problem
\begin{equation}\label{B}
	\begin{cases}
		-\mathscr{L}_p v=f^{\star}, &\text { in }\Omega^\star, \\
		v=0,& \text { on } \partial \Omega^\star.
	\end{cases}
\end{equation}
More precisely, we will prove the following result:

\vspace{0.3cm}

If $u \in W^{1, p}(\Omega,\mathfrak{m})$ is a solution to $(\ref{A})$ and $v \in W^{1, p}([0, r_a), \mathfrak{m}_{N})$ is the solution to $(\ref{B})$, then
\begin{equation}\label{Eq}
	\mathrm{AVR}_{\mathfrak{m}}^{\frac{q}{N}}u^{\star}(x) \leq v(x),\quad \mathfrak{m}_{N}\text{-a.e. }x\in \Omega^\star.
\end{equation}  
If $\mathrm{AVR}_{\mathfrak{m}}^{\frac{q}{N}}u^{\star} = v$ $\mathfrak{m}_{N}$-a.e. in $\Omega^\star$, then $(\mathrm{X},\mathrm{d},\mathfrak{m})$ is a Euclidean metric measure cone.

\vspace{0.3cm}

Now, we can ask an almost rigidity: is a $\mathrm{RCD}(0,N)$ with $N\in (1,\infty)$ and $\mathrm{AVR}_{\mathfrak{m}}>0$ that admits a solution satisfying almost equality (in the sense of $L^p$-norm) in $(\ref{Eq})$ closed to a Euclidean metric measure cone? The answer is yes. In fact, our main result in this note is the following:
\begin{theorem}
	For every $\epsilon>0$, $p>1$, $N>1$, $0<a_1<a_2<\infty$, $b,d \in(0,\infty)$, $0<c_l \leq c_u<\infty$, there exists $\delta=\delta(\epsilon, p, N, a_1,a_2,b, c_l / c_u,d)>0$ such that the following statement holds.
	Assume that
	\begin{itemize}
		\item $(\mathrm{X}, \mathrm{d}, \mathfrak{m},\bar{x})$ is a $\operatorname{RCD}(0, N)$ space with $\mathrm{AVR}_{\mathfrak{m}}\geq b$ and $\Omega=B_R(x) \subset \mathrm{X}$ is an open ball with $\mathfrak{m}(\Omega):=a\in [a_1,a_2]$ and $\mathrm{d}(\bar{x},x)\leq d;$
		\item $f \in W_0^{1, q}(\Omega, \mathfrak{m})$ with $c_l \leq\|f\|_{L^q(\Omega, \mathfrak{m})} \leq\|f\|_{W^{1, q}(\Omega, \mathfrak{m})} \leq c_u;$
		\item $u \in W^{1, p}(\Omega, \mathfrak{m})$ is a solution to $(\ref{A});$
		\item $v \in W^{1, p}([0, r_a), \mathfrak{m}_{N})$ is a solution to $(\ref{B})$, where $r_a=H_N^{-1}(a)$.
	\end{itemize}
	If $\|\mathrm{AVR}_{\mathfrak{m}}^{\frac{q}{N}} u^{\star}-v\|_{L^p([0, r_a), \mathfrak{m}_{N})}<\delta$, then there exists a Euclidean metric measure cone $(\mathrm{Z}, \mathrm{d}_{\mathrm{Z}}, \mathfrak{m}_{\mathrm{Z}},\bar{z})$ such that
	$$
	\mathrm{d}_{\mathrm{pmGH}}((\mathrm{X}, \mathrm{d}, \mathfrak{m},\bar{x}),(\mathrm{Z}, \mathrm{d}_{\mathrm{Z}}, \mathfrak{m}_{\mathrm{Z}},\bar{z}))<\epsilon .
	$$
\end{theorem}
The proof of this theorem relies on a compactness result:  
\begin{theorem}\label{thmB}
	Let $p,N\in (1,\infty)$, let $\{\mathscr{X}_n:=[\mathrm{X}_n, \mathrm{d}_n, \mathfrak{m}_n,\bar{x}_n]\}_{n\in \overline{\mathbb{N}}_+}$ be a sequence of $\mathrm{RCD}(0,N)$ spaces with  $$\mathrm{d}_{\mathrm{pmGH}}(\mathscr{X}_n,\mathscr{X}_\infty)\rightarrow 0,$$  
	and let $\{R_n\}_{n\in \overline{\mathbb{N}}_+}\subset (0,\infty)$ be such that $R_n\rightarrow R_\infty$.

	If the sequence $\mathbb{N}_+\ni n\mapsto f_n\in W_0^{1, p}(B_{R_n}(x_n),\mathfrak{m}_n)$ is such that $\sup_n\|f_n\|_{W^{1, p}(\mathrm{Y}, \mathfrak{m}_n)}<\infty$, then $\{f_n\}_n$ has a subsequence $\{f_{n_k}\}_k$ that $L^p$-strongly converges to some function $f_\infty\in W^{1, p}(\mathrm{Y},\mathfrak{m}_\infty)$.
\end{theorem}
We mention that Theorem $\ref{thmB}$ was already proved in the case $p=2$ in \cite{GMS}. Our proof will use the sequences of covering families and weak local Poincar\'{e} inequalities, inspired by the works in \cite{BK22}.

\section{Preliminaries}
$\quad$ Throughout this note, a metric measure space is a triple $(\mathrm{X},\mathrm{d},\mathfrak{m})$ where $(\mathrm{X},\mathrm{d})$ is a complete and separable metric space, and $\mathfrak{m}\neq 0$ is a non-negative Borel measure that is finite on balls. Let $p\in (1,\infty)$ and let $q$ be the conjugate exponent by the relation
$p^{-1}+q^{-1}=1$.
\subsection{Sobolev functions}

The Cheeger energy (see \cite{AGS13}, \cite{AGS14}, \cite{Ch99}, \cite{Sh00}) associated to a metric measure space $(\mathrm{X},\mathrm{d},\mathfrak{m})$ is the convex and lower semi-continuous functional defined on $L^p(\mathrm{X},\mathfrak{m})$ as
$$
\operatorname{Ch}_p(f):=\frac{1}{p} \inf \bigg\{\liminf_{k\rightarrow \infty} \int \operatorname{lip}^p(f_k) \mathrm{d}\mathfrak{m}:\{f_k\}_k \subset \operatorname{Lip}_{\mathrm{bs}}(\mathrm{X}), f_k \rightarrow f \text { in } L^p(\mathrm{X},\mathfrak{m})\bigg\},
$$
where $\operatorname{lip}(f)$ is the so called local Lipschitz constant 
$$
\operatorname{lip}(f)(x):=\limsup _{y \rightarrow x} \frac{|f(y)-f(x)|}{\mathrm{d}(y, x)},
$$
which has to be understood to be $0$ if $x$ is an isolated point. The finiteness domain of the Cheeger energy is denoted by $W^{1,p}(\mathrm{X},\mathfrak{m})$ and is endowed with the complete norm $\|f\|_{W^{1,p}(\mathrm{X},\mathfrak{m})}:=\big[\|f\|_{L^p}^p+p \mathrm{Ch}_p(f)\big]^{1/p}$. It is possible to identify a canonical object $|D f|_{*,p} \in L^p(\mathrm{X},\mathfrak{m})$, called minimal $p$-relaxed slope, providing the integral representation
$$
\operatorname{Ch}_p(f)=\frac{1}{p} \int|D f|_{*,p}^p \mathrm{d}\mathfrak{m},\quad \text { for any } f \in W^{1,p}(\mathrm{X},\mathfrak{m}).
$$
Any metric measure space on which $\operatorname{Ch}_2$ is a quadratic form is said to be infinitesimally Hilbertian (\cite{Gi15}). 

Given an open subset $\Omega\subset \mathrm{X}$, the space $W_0^{1,p}(\Omega,\mathfrak{m})$ is defined by the closure of $\mathrm{Lip}_{\mathrm{c}}(\Omega)$ with respect to the norm of $W^{1,p}(\mathrm{X},\mathfrak{m})$ and the space $\hat{W}_0^{1, p}(\Omega,\mathfrak{m})$ is defined by the set of all $f \in W^{1, p}(\mathrm{X},\mathfrak{m})$ such that $f=0$ $\mathfrak{m}$-a.e. in $\Omega^c$.

\subsection{Differential calculus} 
Let $L^0(\mathrm{X},\mathfrak{m})$ be the vector space of all equivalence classes up to $\mathfrak{m}$-negligible set of Borel functions on $\mathrm{X}$ equipped with the topology of local convergence in measure. We recall the algebraic notion of a normed module over $(\mathrm{X},\mathrm{d},\mathfrak{m})$ and discuss a non-smooth differential calculus on metric measure spaces as in \cite{Gi18}. 

\begin{definition}[$L^0(\mathrm{X},\mathfrak{m})$-normed module]
	We call an $L^0(\mathrm{X},\mathfrak{m})$-normed module the quadruple $(\mathscr{M}^0, \tau, \cdot,|\cdot|)$, where\\
	$\mathrm{i})$ $(\mathscr{M}^0, \tau)$ is a topological vector space.\\
	$\mathrm{ii})$ The multiplication by $L^0$-functions $\cdot:L^0(\mathrm{X},\mathfrak{m}) \times \mathscr{M}^0 \rightarrow \mathscr{M}^0$ is a bilinear map satisfying
	\begin{align*}
		f \cdot(g \cdot v)=(f g) \cdot v ,\quad \hat{1} \cdot v=v,\quad \text { for any } f, g \in L^{0}(\mathrm{X},\mathfrak{m}) \text { and } v \in \mathscr{M}^0, 
	\end{align*}
	where $\hat{1}$ denotes the equivalence class up to $\mathfrak{m}$-negligible set of the function identically equal to $1$.\\
	$\mathrm{iii})$ The pointwise norm $|\cdot|: \mathscr{M}^0 \rightarrow L^0(\mathrm{X},\mathfrak{m})$ satisfies
	\begin{align*}
		|v| \geq 0 \quad \mathfrak{m} \text {-a.e., }   \quad
		|f \cdot v|=|f||v| , 
	\end{align*}
	for any $f \in L^0(\mathrm{X},\mathfrak{m})$ and $v \in \mathscr{M}^0$.\\
	$\mathrm{iv})$ The function $\mathrm{d}_{\mathscr{M}^0}:\mathscr{M}^0 \times \mathscr{M}^0 \rightarrow[0,\infty)$, defined by
	$$\mathrm{d}_{\mathscr{M}^0}(v, w):=\int|v-w| \wedge 1 \mathrm{d}\mathfrak{m}^{\prime},
	$$
	is a complete distance on $\mathscr{M}^0$ inducing the topology $\tau$. Here, $\mathfrak{m}^{\prime}$ is a Borel probability measure on $\mathrm{X}$ such that $\mathfrak{m} \ll \mathfrak{m}^{\prime} \ll \mathfrak{m}$. $($Such an $\mathfrak{m}^{\prime}$ must exist.$)$
\end{definition} 
An $L^0(\mathrm{X},\mathfrak{m})$-normed module $\mathscr{M}^0$ is called a Hilbert module if 
$$
2|v|^2+2|w|^2=|v+w|^2+|v-w|^2,
$$
for any $v,w\in \mathscr{M}^0$. 

\begin{proposition}\label{prop1}
	Let $\mathscr{M}^0$ be a Hilbert module. Then the formula
	$$
	\langle v, w\rangle:=\frac{1}{2}\big(|v+w|^2-|v|^2-|w|^2\big) \in L^0(\mathrm{X},\mathfrak{m})
	$$
	defines an $L^{0}(\mathrm{X},\mathfrak{m})$-bilinear map $\langle\cdot, \cdot\rangle: \mathscr{M}^0 \times \mathscr{M}^0 \rightarrow L^0(\mathrm{X},\mathfrak{m})$, called pointwise scalar product, which satisfies
	\begin{align*}
		\langle v, w\rangle=\langle w, v\rangle, \quad
		|\langle v, w\rangle| \leq|v||w|\ \mathfrak{m}\text{-a.e.,} \quad
		\langle v, v\rangle=|v|^2,
	\end{align*}
	for any $v, w \in\mathscr{M}^0$.
\end{proposition}
\begin{proof}
	All properties can be obtained by suitably adapting the proof of the analogous statements for Hilbert spaces, apart from the $L^0(\mathrm{X},\mathfrak{m})$-bilinearity of $\langle\cdot, \cdot\rangle$, which can be shown by using the fact that $\langle\cdot, \cdot\rangle$ is local and continuous with respect to both entries by its very construction.
\end{proof}
\begin{definition}[Dual of $L^0(\mathrm{X},\mathfrak{m})$-normed module]
	Let $\mathscr{M}^0$ be an $L^0(\mathrm{X},\mathfrak{m})$-normed module. Then, we define its dual module $(\mathscr{M}^0)^*$ as
	\begin{align*}
		(\mathscr{M}^0)^* :=\big\{L: \mathscr{M}^0 \rightarrow L^0(\mathrm{X},\mathfrak{m}) & | L \text { is a continuous linear operator satisfying }  \\
		&\quad  L(f\cdot v)=f L(v) \text { for any } v \in \mathscr{M}^0\text{ and }f \in L^{0}(\mathrm{X},\mathfrak{m})\big\},
	\end{align*}
	equipped with the following operations
	\begin{align*}
		(\alpha L+L^\prime)(v)&:=\alpha L(v)+L^\prime(v),\\
		(f\cdot L)(v)&:=fL(v),\\
		|L|_*&:=\operatorname{ess-sup}_{v \in \mathscr{M}^0:|v| \leq 1\ \mathfrak{m}\text{-a.e. }} L(v),
	\end{align*}
	for any $L,L^\prime\in (\mathscr{M}^0)^*$, $\alpha\in \mathbb{R}$, $f\in L^0(\mathrm{X},\mathfrak{m})$ and $v\in \mathscr{M}^0$.
\end{definition}
One can check that $(\mathscr{M}^0)^*$ has a natural $L^0(\mathrm{X},\mathfrak{m})$-normed module structure.
\begin{proposition}[\cite{GN22}]
	Let $(\mathrm{X}, \mathrm{d}, \mathfrak{m})$ be a $\operatorname{RCD}(K, N)$ space for some $K \in \mathbb{R}$, $N \in[1, \infty]$. Then there exists a triple $(L^0(T^*\mathrm{X}),\mathrm{d},\nabla)$ where $L^0(T^*\mathrm{X})$ is a Hilbert module, and $\mathrm{d}: \cup_{p \in(1, \infty)} W^{1, p}(\mathrm{X},\mathfrak{m}) \rightarrow L^0(T^*\mathrm{X})$ and $\nabla: \cup_{p \in(1, \infty)} W^{1, p}(\mathrm{X},\mathfrak{m}) \rightarrow L^0(T \mathrm{X}):=(L^0(T^*\mathrm{X}))^*$ are such that for any $p \in(1, \infty)$ it holds that\\
	$\mathrm{i})$ $\mathrm{d}|_{W^{1, p}(\mathrm{X},\mathfrak{m})}$ and $\nabla|_{W^{1, p}(\mathrm{X},\mathfrak{m})}$ are linear.\\
	$\mathrm{ii})$ For any $f \in W^{1, p}(\mathrm{X},\mathfrak{m})$, $\mathrm{d} f(\nabla f)=|\nabla f|_*^2=|\mathrm{d} f|^2=|D f|_{*,p}^2$.\\
	$\mathrm{iii})$ The spaces $\{\mathrm{d} f: f \in W^{1, p}(\mathrm{X},\mathfrak{m})\}$ and $\{\nabla f: f \in W^{1, p}(\mathrm{X},\mathfrak{m})\}$ generate $L^0(T^* \mathrm{X})$ and $L^0(T \mathrm{X})$ as modules, respectively.
\end{proposition}

\begin{definition}\label{defSobo}
	Let $\Omega$ be an open subset of $\mathrm{X}$. We say that $f\in L^p(\Omega,\mathfrak{m})$ belongs to $W^{1,p}(\Omega,\mathfrak{m})$ if \\
	$\mathrm{a})$ for any $\eta\in \mathrm{Lip}_{\mathrm{c}}(\Omega)$, it holds that $\eta f\in W^{1,p}(\mathrm{X},\mathfrak{m});$\\
	$\mathrm{b})$ there exists a non-negative function $G\in L^p(\Omega,\mathfrak{m})$ such that for any $\eta\in \mathrm{Lip}_{\mathrm{c}}(\Omega)$, we have that
	$$
	|D (\eta f)|_{*,p}\leq G,\quad \mathfrak{m}\text{-a.e. on }\{\eta=1\}.
	$$
\end{definition}

\vspace{0.5cm}

Given a $\operatorname{RCD}(K, N)$ space with $K \in \mathbb{R},N \in[1, \infty)$ and an open subset $\Omega\subset \mathrm{X}$, for any $f,g\in W^{1,p}(\Omega,\mathfrak{m})$ we define $|\nabla f|\in L^p(\Omega,\mathfrak{m})$ and $\langle\nabla f, \nabla g\rangle\in L^{p/2}(\Omega,\mathfrak{m})$ via
$$
|\nabla f|=|D (\eta f)|_{*,p}\text{ and } \langle\nabla f, \nabla g\rangle:=\langle\nabla(\eta f), \nabla (\eta g)\rangle,\quad\mathfrak{m}\text{-a.e. on }\{\eta=1\},
$$
for any $\eta\in \mathrm{Lip}_{\mathrm{c}}(\Omega)$. Notice that the property $\mathrm{a})$, together with the locality property of the minimal $p$-relaxed slopes, guarantees that $|\nabla f|$ and $\langle\nabla f, \nabla g\rangle$ are well-defined.

\begin{definition}[The local $p$-Laplacian and Poisson problem on $\Omega$]
	Let $(\mathrm{X}, \mathrm{d}, \mathfrak{m})$ be a $\operatorname{RCD}(K, N)$ space for some $K \in \mathbb{R}$, $N \in[1, \infty)$, let $\Omega$ be an open subset of $\mathrm{X}$ and let $f\in L^q(\Omega,\mathfrak{m})$. 
	
	The $p$-Laplacian on $\Omega$ is an operator $\mathscr{L}_p$ on $W^{1, p}(\Omega,\mathfrak{m})$ defined as the follows. For each function $u \in W^{1, p}(\Omega,\mathfrak{m})$, its $p$-Laplacian $\mathscr{L}_p u$ is a linear functional acting on $W_0^{1, p}(\Omega,\mathfrak{m})$ given by
	$$
	\mathscr{L}_p u(\psi):=-\int_{\Omega} \langle\nabla u, \nabla \psi\rangle|\nabla u|^{p-2} \mathrm{d}\mathfrak{m}, \quad \text{ for every } \psi \in W_0^{1, p}(\Omega,\mathfrak{m}).
	$$

	We say that a function $u\in W^{1,p}(\Omega,\mathfrak{m})$ is a solution to the Poisson problem 
	\begin{equation}\label{eq1}
		\begin{cases}
			-\mathscr{L}_p u=f, &\text{ in }\Omega,\\
			u=0, &\text{ on }\partial\Omega,
		\end{cases}
	\end{equation}
	if $u\in W_0^{1,p}(\Omega,\mathfrak{m})$ and for any $\psi\in W_0^{1,p}(\Omega,\mathfrak{m})$, it holds that
	$$
	-\mathscr{L}_p u(\psi)=\int_\Omega \langle\nabla u, \nabla \psi\rangle|\nabla u|^{p-2} \mathrm{d}\mathfrak{m}=\int_\Omega f \psi \mathrm{d}\mathfrak{m}.
	$$

\end{definition}

\subsection{Functions of bounded variation and Coarea formula}
We assume that the reader is familiar with the theory of functions of bounded variation and sets of finite perimeter in metric measure spaces developed in \cite{M03} and in the more recent \cite{AD14}. We recall now the main notions.
\begin{definition}[$\infty$-test plan]
	A measure $\pi\in \mathscr{P}(\mathrm{C}([0,1],\mathrm{X}))$ is said to be an $\infty$-test plan if\\
	$\mathrm{i})$ there exists $C>0$ such that $(\mathrm{e}_t)_\sharp \pi\leq C\mathfrak{m}$ for every $t\in [0,1];$\\
	$\mathrm{ii})$ $\gamma\mapsto \mathrm{Lip}(\gamma)\in {L}^\infty(\mathrm{C}([0,1],\mathrm{X}),\pi)$.
\end{definition}
The least constant $C>0$ that can be chosen in $\mathrm{i})$ is called compression constant of $\pi$ and is denoted by $\mathrm{Comp}(\pi)$.

\begin{definition}\label{codef1}
	The $\mathrm{BV}$ space $\mathrm{BV}_{\mathrm{w}}(\mathrm{X},\mathfrak{m})$ is defined as the space of all Borel functions $f: \mathrm{X} \rightarrow \mathbb{R}$ that satisfy the following properties:\\
	$\mathrm{i})$ for every $\infty$-test plan $\pi$, it holds that for $\pi$-a.e. $\gamma$, $f\circ {\gamma}\in \mathrm{BV}((0,1))$ and
	$$
	\liminf_{k\rightarrow \infty}\bigg[\fint_0^{1/k}|f(\gamma_s)-f(\gamma_0)|\mathrm{d}s+\fint_{1-1/k}^{1}|f(\gamma_s)-f(\gamma_1)|\mathrm{d}s\bigg]=0;
	$$
	$\mathrm{ii})$ there exists a finite Borel measure $\mu$ on $\mathrm{X}$ such that for every $\infty$-test plan $\pi$,
	\begin{equation}\label{coeq1}
		\int {\gamma}_{\sharp}|D(f \circ {\gamma})| \mathrm{d} {\pi}(\gamma) \leq \mathrm{Comp}(\pi) \|\operatorname{Lip}(\cdot)\|_{{L}^{\infty}({\pi})} \mu .
	\end{equation}
\end{definition}
Associated to this notion, there is the concept of weak total variation $|Df|_w$, defined as the least measure $\mu$ satisfying $(\ref{coeq1})$ for every $\infty$-test plan ${\pi}$.

Let $E\subset \mathrm{X}$ be a Borel set, we say that $E$ is a set of locally finite perimeter if $|\mathrm{D} \chi_E|_w(A)<\infty$ for every $A$ bounded open subset of $\mathrm{X}$ and we say that $E$ is a set of finite perimeter if $|\mathrm{D} \chi_E|_w(\mathrm{X})<\infty$. We denote $|\mathrm{D} \chi_E|_w(\cdot)$ also with $\operatorname{Per}(E, \cdot)$. For simplicity, $\operatorname{Per}(E):=\operatorname{Per}(E, \mathrm{X})$.
\begin{proposition}[\cite{AD14,D14}]\label{coprop1}
	Let $f\in L^1(\mathrm{X},\mathfrak{m})\cap \mathrm{BV}_{\mathrm{w}}(\mathrm{X},\mathfrak{m})$. Then 
	\begin{align*}
		|Df|_w(\mathrm{X})=\inf\bigg\{\liminf_{n\rightarrow \infty}\int \mathrm{lip}(f_n)\mathrm{d}\mathfrak{m}:\{f_n\}_n\subset \operatorname{Lip}_{\mathrm{bs}}(\mathrm{X}),\|f_n-f\|_{{L}^1}\rightarrow 0\bigg\},
	\end{align*}
	and for every open set $B\subset \mathrm{X}$, 
	\begin{align*}
		|Df|_w(B)=\inf\bigg\{\liminf_{n\rightarrow \infty}\int_B \mathrm{lip}(f_n)\mathrm{d}\mathfrak{m}:\{f_n\}_n\subset \operatorname{Lip}_{\mathrm{loc}}(B)\cap L^1(B),\|f_n-f\|_{{L}^1(B)}\rightarrow 0\bigg\}.
	\end{align*}
\end{proposition}

\begin{lemma}[\cite{GH14}]\label{colem1}
	Let $(\mathrm{X},\mathrm{d},\mathfrak{m})$ be a $\mathrm{RCD}(K,\infty)$ space and let $p\in (1,\infty)$. Let $f\in \mathrm{BV}_{\mathrm{w}}(\mathrm{X},\mathfrak{m})$ be such that $|Df|_w\ll \mathfrak{m}$, and denote by $|Df|_1$ the density function of $|Df|_w$ with respect to $\mathfrak{m}$. If $f\in  {L}_{\mathrm{loc}}^p(\mathrm{X},\mathfrak{m}),|Df|_1\in {L}^p(\mathrm{X},\mathfrak{m})$, then $f\in \mathrm{S}^{p}(\mathrm{X})$ and $|Df|_{w,p}=|Df|_1$ $\mathfrak{m}$-a.e.
\end{lemma}

\begin{theorem}[Coarea formula]
	Let $(\mathrm{X},\mathrm{d},\mathfrak{m})$ be a $\mathrm{RCD}(K,\infty)$ space and let $p\in (1,\infty)$. If $f \in W^{1, p}(\mathrm{X},\mathfrak{m})$ is a non-negative function with $|Df|_{*,p}\in L^1(\mathrm{X},\mathfrak{m})$, then for any Borel function $g: \mathrm{X} \rightarrow[0,\infty)$ and $t\geq 0$, it holds that
	\begin{equation}\label{coeq4}
		\int_{\{f>t\}} g|Df|_{*,p} \mathrm{d}\mathfrak{m}=\int_t^{\infty} \int g \mathrm{~d} \mathrm{Per}(\{f>r\}) \mathrm{~d} r.
	\end{equation}
\end{theorem}
\begin{proof}
	Although this result is well-known, we present a proof here for the sake of completeness.
	
	\textbf{Step 1}: we prove that for every bounded open set $B\subset\mathrm{X}$, the mapping $[0,\infty)\ni t\mapsto\operatorname{Per}(\{f>t\},B)\in[0,\infty]$ is Borel. 
	
	To this aim, put $N=\{t\geq 0:\mathfrak{m}(\{f=t\}\cap B)>0\}$. Since $B$ has finite measure, it is not difficult to prove that $N$ is a countable set. Thus, it suffices to prove that the mapping $[0,\infty)\setminus N\ni t\mapsto\operatorname{Per}(\{f>t\},B)\in[0,\infty]$ is Borel.
	
	Let $t,t_1,t_2,\cdots\in[0,\infty)\setminus N$ be such that $t_n\rightarrow t$. Then $\|\chi_{\{f>t_n\}}-\chi_{\{f>t\}}\|_{L^1(B)}\rightarrow 0$. By the lower semi-continuity, we have that
	$$
	\operatorname{Per}(\{f>t\},B)\leq \liminf_{n\rightarrow\infty}\operatorname{Per}(\{f>t_n\},B),
	$$
	which means that the mapping $[0,\infty)\setminus N\ni t\mapsto\operatorname{Per}(\{f>t\},B)\in[0,\infty]$ is lower semi-continuous, and hence Borel, as desired.
	
	\textbf{Step 2}: Let $f\in W^{1,p}(\mathrm{X},\mathfrak{m})$ be a non-negative function, and let $B\subset\mathrm{X}$ be a bounded open set. We prove that $f\in \mathrm{BV}_{\mathrm{w}}(\mathrm{X},\mathfrak{m})$, $|Df|_w\ll \mathfrak{m}$ and for any $s\geq 0$,
	$$
	\int_s^{\infty}\operatorname{Per}(\{f>t\},B)\,\mathrm{d}t\leq \int_{\{f>s\}\cap B}|D f|_{*,p}\mathrm{d}\mathfrak{m}.
	$$
	
	Fix $\bar{x}\in\mathrm{X}$ and for any $R>0$, choose a 1-Lipschitz function $\eta_R\in\operatorname{Lip}_{\mathrm{bs}}(\mathrm{X})$ such that $\eta_R=1$ on $B_R(\bar{x})$ and $\eta_R=0$ on $B_{R+1}(\bar{x})^c$. For any $n\in\mathbb{N}_{+}$, we set $f_n:=\eta_n f\in W^{1,p}(\mathrm{X},\mathfrak{m})$. It is easy to check that
	$$
	\|f_{n}-f\|_{L^p}\xrightarrow{n}0,\quad\||D f_n|_{w,p}-|D f|_{*,p}\|_{L^p}\xrightarrow{n}0.
	$$
	
	Now, for each $n$, since $f_n$ has bounded support, there exists a sequence of non-negative functions $\{f_{n,m}\}_m\in\operatorname{Lip}_{\mathrm{bs}}(\mathrm{X})$ that have uniformly bounded support such that
	$$
	\|f_{n,m}-f_n\|_{L^p}\xrightarrow{m}0,\quad\|\operatorname{lip}(f_{n,m})-|D f_n|_{*,p}\|_{L^p}\xrightarrow{m}0,
	$$
	which implies
	$$
	\|f_{n,m}-f_n\|_{L^1}\xrightarrow{m}0,\quad\|\operatorname{lip}(f_{n,m})-|D f_n|_{*,p}\|_{L^1}\xrightarrow{m}0.
	$$
	Therefore, $f_n\in\mathrm{BV}_{\mathrm{w}}(\mathrm{X},\mathfrak{m})$ and $|D f_n|_1=|D f_n|_{w,p}=|D f_n|_{*,p}$ $\mathfrak{m}$-a.e. by Proposition $\ref{coprop1}$, Lemma $\ref{colem1}$ and the equivalence of minimal $p$-weak upper gradients and minimal $p$-relaxed slopes(see \cite{AGS13}). 
	
	To prove $f\in \mathrm{BV}_{\mathrm{w}}(\mathrm{X},\mathfrak{m})$, let $\pi$ be an $\infty$-test plan, and let $B$ be a bounded open subset of $\mathrm{X}$. The fact that $\|f_n-f\|_{L^1(B)}\rightarrow 0$ implies that, up to a not relabeled subsequence, $\|f_n \circ \gamma-f \circ\gamma \|_{L^1(\gamma^{-1}(B))} \rightarrow 0$ for $\pi$-a.e. $\gamma$ (see \cite[Lemma 5.7]{AD14} for the details). Then
	\begin{align*}
		\int |D(f\circ\gamma)|(\gamma^{-1}(B))\mathrm{d}\pi(\gamma)&\leq \liminf_{n\rightarrow \infty}\int |D(f_n\circ\gamma)|(\gamma^{-1}(B))\mathrm{d}\pi(\gamma) 	\\
		&\leq \mathrm{Comp}(\pi) \|\operatorname{Lip}(\cdot)\|_{{L}^{\infty}({\pi})} \liminf_{n\rightarrow \infty}|Df_n|_w(B)\\
		&=\mathrm{Comp}(\pi) \|\operatorname{Lip}(\cdot)\|_{{L}^{\infty}({\pi})} \liminf_{n\rightarrow \infty}\int_B |Df_n|_{*,p}\mathrm{d}\mathfrak{m}\\
		&\leq \mathrm{Comp}(\pi) \|\operatorname{Lip}(\cdot)\|_{{L}^{\infty}({\pi})} \liminf_{n\rightarrow \infty}\int_B |Df|_{*,p}+|f|\chi_{ B_n(\bar{x})^c}\mathrm{d}\mathfrak{m}\\
		&=\mathrm{Comp}(\pi) \|\operatorname{Lip}(\cdot)\|_{{L}^{\infty}({\pi})} \int_B |Df|_{*,p}\mathrm{d}\mathfrak{m}.
	\end{align*}
	Hence, $|Df|_w\leq |Df|_{*,p}\mathfrak{m}$. On the other hand, the condition i) in the definition $\ref{codef1}$ holds by the construction of $\{f_n\}$. Therefore, we obtain that $f\in \mathrm{BV}_{\mathrm{w}}(\mathrm{X},\mathfrak{m})$.
	
	Next, fix $n,m$, and define $g_{n,m}:[0,\infty)\rightarrow[0,\infty)$ by
	$$
	g_{n,m}(t)=\int_{\{f_{n,m}\leq t\}\cap B}\operatorname{lip}(f_{n,m})\mathrm{d}\mathfrak{m}.
	$$
	Note that $g_{n,m}$ is a bounded non-decreasing function, hence differentiable almost everywhere. Let $t>0$ be a differentiability point of $g_{n,m}$ such that $\mathfrak{m}(\{f_{n,m}=t\}\cap B)=0$. For any $r>0$, we define $h_{n,m,r}:\mathbb{R}\rightarrow\mathbb{R}$ by
	$$
	h_{n,m,r}(s)=\begin{cases}
		0,&s\leq  t,\\
		\text{linear},&t<s<t+1/r,\\
		1,&s\geq  t+1/r.
	\end{cases}
	$$
	Define $v_{n,m,r}=h_{n,m,r}\circ f_{n,m}|_B$. We have that
	$$
	\int_B|v_{n,m,r}-\chi_{\{f_{n,m}>t\}}| \mathrm{d}\mathfrak{m}\leq \mathfrak{m}(\{t<f_{n,m}<t+1/r\}\cap B)\xrightarrow{r\rightarrow\infty}0,
	$$
	and
	$$
	\int_B\operatorname{lip}(v_{n,m,r})\mathrm{d}\mathfrak{m}\leq r\int_{\{t<f_{n,m}\leq t+1/r\}\cap B}\operatorname{lip}(f_{n,m})\mathrm{d}\mathfrak{m}=r(g_{n,m}(t+1/r)-g_{n,m}(t))\rightarrow g_{n,m}^{\prime}(t),
	$$
	as $r\rightarrow\infty$. Therefore, $\operatorname{Per}(\{f_{n,m}>t\},B)\leq g_{n,m}^{\prime}(t)$, and hence for any $s\geq 0$,
	\begin{equation}\label{coeq5}
		\int_s^{\infty}\operatorname{Per}(\{f_{n,m}>t\},B)\mathrm{d}t\leq\int_{\{f_{n,m}>s\}\cap B}\operatorname{lip}(f_{n,m})\mathrm{d}\mathfrak{m}.
	\end{equation}
	
	Since
	$$
	\int_{\mathbb{R}}|\chi_{\{f_{n,m}>t\}\cap B}(x)-\chi_{\{f_n>t\}\cap B}(x)|\mathrm{d}t=|f_{n,m}(x)-f_n(x)|, \text{ for every }x\in B,
	$$
	it follows that
	$$
	\int_B|f_{n,m}-f_n|\mathrm{d}\mathfrak{m}=\int_{\mathbb{R}}\int|\chi_{\{f_{n,m}>t\}\cap B}-\chi_{\{f_n>t\}\cap B}|\mathrm{d}\mathfrak{m}\mathrm{d}t.
	$$
	Combining this with $\|f_{n,m}-f_n\|_{L^1}\xrightarrow{m}0$, we know that there exists a sequence $\{m_k\}_k$ such that $\|\chi_{\{f_{n,m_k}>t\}}-\chi_{\{f_n>t\}}\|_{L^1(B)}\xrightarrow{k}0$ for a.e. $t>0$. Thus, by the lower semi-continuity, we have that $\operatorname{Per}(\{f_{n}>t\},B)\leq \liminf_{k\rightarrow\infty}\operatorname{Per}(\{f_{n,m_k}>t\},B)$ for a.e. $t>0$. Therefore, by Fatou's lemma, $(\ref{coeq5})$, $\|\operatorname{lip}(f_{n,m})-|D f_n|_{*,p}\|_{L^p}\xrightarrow{m}0$ and $\|\chi_{\{f_{n,m_k}>s\}}-\chi_{\{f_n>s\}}\|_{L^\frac{p}{p-1}(B)}\xrightarrow{k}0$ for a.e. $s>0$, we deduce that
	\begin{align*}
		\int_s^{\infty}\operatorname{Per}(\{f_{n}>t\},B)\mathrm{d}t&\leq \liminf_{k\rightarrow\infty}\int_s^{\infty}\operatorname{Per}(\{f_{n,m_k}>t\},B)\mathrm{d}t\\
		&=\liminf_{k\rightarrow\infty}\int \chi_{\{f_{n,m_k}>s\}\cap B}\operatorname{lip}(f_{n,m_k})\mathrm{d}\mathfrak{m}\\
		&=\int_{\{f_n>s\}} \chi_{B}|D f_n|_{*,p}\mathrm{d}\mathfrak{m},\text{ for a.e. }s>0.
	\end{align*}
	This inequality holds for all $s\geq 0$, due to the absolute continuity of the integral. By a similar argument, we can also prove that
	$$
	\int_s^{\infty}\operatorname{Per}(\{f>t\},B)\mathrm{d}t\leq\int_{\{f>s\}} \chi_{B}|D f|_{*,p}\mathrm{d}\mathfrak{m},\text{ for every }s\geq 0.
	$$
	
	\textbf{Step 3}: Let $f\in W^{1,p}(\mathrm{X},\mathfrak{m})$ be a non-negative function, and let $B\subset\mathrm{X}$ be a bounded open set. We prove that for any $s\geq 0$,
	\begin{equation}\label{coeq2}
		\int_s^{\infty}\operatorname{Per}(\{f>t\},B)\,\mathrm{d}t\geq \int_{\{f>s\}\cap B}|D f|_{*,p}\mathrm{d}\mathfrak{m}.
	\end{equation}
	
	Fix $h,n\in\mathbb{N}_{+}$ and for any $j=1,\cdots,n h$, there exists $t_{j,h}\in({(j-1)}/{h}+s,{j}/{h}+s)$ such that
	$$
	\frac{1}{h}\operatorname{Per}(\{f>t_{j,h}\},B)\leq\int_{(j-1)/{h}+s}^{{j}/{h}+s}\operatorname{Per}(\{f>t\},B)\,\mathrm{d}t.
	$$
	Define $f_{n,h}:\mathrm{X}\rightarrow[0,\infty)$ by
	$$
	f_{n, h}=s+\frac{1}{h} \sum_{j=1}^{n h} \chi_{\{f>t_{j, h}\}}.
	$$
	Then we have that
	\begin{align*}
		|D f_{n,h}|_w(B)&\leq\frac{1}{h}\sum_{j=1}^{n h} \operatorname{Per}(\{f>t_{j,h}\},B)\\ &\leq\int_s^{s+n}\operatorname{Per}(\{f>t\},B)\,\mathrm{d}t\leq\int_s^{\infty}\operatorname{Per}(\{f>t\},B)\,\mathrm{d}t.
	\end{align*}
	To prove $(\ref{coeq2})$, it suffices to prove that for all $n$,
	\begin{equation}\label{coeq3}
		\|f_{n,h}-f\wedge(n+s)\vee s\|_{L^1(B)}\xrightarrow{h\rightarrow\infty}0,
	\end{equation}
	Indeed, if this is right, then by the lower semi-continuity, the locality of minimal $p$-weak upper gradient and Lemma $\ref{colem1}$ (because $f\in \mathrm{BV}_{\mathrm{w}}(\mathrm{X},\mathfrak{m})$ implies $f\wedge(n+s)\vee s\in \mathrm{BV}_{\mathrm{w}}(\mathrm{X},\mathfrak{m})$), we have 
	$$
	\int_{\{s<f<n+s\}\cap B} |Df|_{w,p}\mathrm{d}\mathfrak{m}=|D [f\wedge(n+s)\vee s]|_w(B)\leq \int_s^{\infty}\operatorname{Per}(\{f>t\},B)\,\mathrm{d}t.
	$$
	By letting $n\rightarrow \infty$, we will obtain the desired inequality $(\ref{coeq2})$.
	
	Finally, we check that $(\ref{coeq3})$: put
	$$
	F_{i,h}=\{t_{i,h}<f\leq t_{i+1,h}\}, i=1,\cdots,n h-1,\quad F_{n h,h}=\{f>t_{n h,h}\}.
	$$
	Then we have
	$$
	f_{n, h}=s+\frac{1}{h} \sum_{j=1}^{n h} \sum_{i=j}^{n h} \chi_{F_{i, h}}=s+\frac{1}{h} \sum_{i=1}^{n h} i \chi_{F_{i, h}}
	$$
	and thus
	\begin{align*}
		\int_B|f \wedge(n+s) \vee s-f_{n, h}| \mathrm{d} \mathfrak{m}&=  \sum_{i=1}^{n h-1} \int_{B\cap F_{i, h}}\bigg|f-\frac{i}{h}-s\bigg| \mathrm{d} \mathfrak{m}+\int_{B \cap\{f \leq t_{1, h}\}}|f \vee s-s| \mathrm{d} \mathfrak{m} \\
		& \quad+\int_{B \cap\left\{f>t_{n h, h}\right\}}|f \wedge(n+s)-(n+s)| \mathrm{d} \mathfrak{m} \\
		&\leq  \frac{1}{h} \mathfrak{m}(B)\xrightarrow{h\rightarrow\infty}0.
	\end{align*}
	
	Combining Step 2 and Step 3, we know that $(\ref{coeq4})$ holds for $g=\chi_B$ if $B$ is bounded and open. By the $\pi$-$\lambda$ Lemma and the fact that any non-negative Borel function can be represented as the pointwise limit of simple Borel functions, we have that for any non-negative Borel function $g: \mathrm{X} \rightarrow[0,\infty)$, $[0,\infty)\ni t\mapsto \int g \mathrm{~d} \mathrm{Per}(\{f>t\})$ is Borel and $(\ref{coeq4})$ holds.
\end{proof}

Next, we recall the isoperimetric inequality in CD spaces, which was obtained in \cite{BK23} for the Minkowski content. This result can be extended to the perimeter through an approximation argument (see \cite{ADG17}).

\begin{theorem}\label{thmiso}
	Let $(\mathrm{X}, \mathrm{d}, \mathfrak{m})$ be a $\operatorname{CD}(0, N)$ space with $N>1$. Then for every set of finite perimeter $E \subset \mathrm{X}$ with $\mathfrak{m}(E)<\infty$, it holds that
	\begin{equation}\label{iso}
		\mathrm{Per}(E) \geq N \omega_N^{\frac{1}{N}} \mathrm{AVR}_{\mathfrak{m}}^{\frac{1}{N}} \mathfrak{m}(E)^{\frac{N-1}{N}}.
	\end{equation}
\end{theorem}

\subsection{Rearrangements and symmetrizations}
Let $\Omega \subset \mathrm{X}$ be an open subset with $\mathfrak{m}(\Omega)\in (0,\infty)$.
\begin{definition}[Distribution Function and Decreasing Rearrangement]
	Let $u: \Omega \rightarrow \mathbb{R}$ be a Borel function. We define $\mu=\mu_u:[0,\infty) \rightarrow[0, \mathfrak{m}(\Omega)]$, the distribution function of $u$, as
	$$
	\mu(t):=\mathfrak{m}(\{|u|>t\})
	$$

	For $u$ and $\mu$ as above, we define $u^{\sharp}:[0, \mathfrak{m}(\Omega)] \rightarrow[0,\infty]$, the decreasing rearrangement of $u$, as
	$$
	u^{\sharp}(s):= \begin{cases}\operatorname{ess-sup} |u|, & \text { if } \mathrm{s}=0, \\ \inf \{t\geq 0: \mu(t)<s\}, & \text { if } \mathrm{s}>0 .\end{cases}
	$$
\end{definition}
It is not difficult to check that $u^{\sharp}$ is non-increasing and left-continuous, and $u^{\sharp}$ is right-continuous at $0$.

Given $N \in(1, \infty)$ and $u: \Omega \rightarrow \mathbb{R}$ a Borel function, we define the Schwarz symmetrization as follows: first, we consider $r>0$ such that $\mathfrak{m}(\Omega)=\mathfrak{m}_{N}([0, r])$, then we define the Schwarz symmetrization $u_{ N}^{\star}=u^{\star}:[0, r] \rightarrow[0, \infty]$ as
$$
u^{\star}(x):=u^{\sharp}(\mathfrak{m}_{N}([0, x])), \quad \text{ for every } x \in[0, r] .
$$

We state here a collection of useful facts concerning the decreasing rearrangement and Schwarz symmetrization of a function. 
\begin{proposition}[Some properties of $u^{\sharp}$ and $u^\star$]\label{propre}
	With the notations above, we have\\
	$(\mathrm{a})$ $u, u^{\sharp}$ and $u^{\star}$ are equi-measurable, in the sense that
	$$
	\mathfrak{m}(\{|u|>t\})=\mathcal{L}^1(\{u^{\sharp}>t\})=\mathfrak{m}_{ N}(\{u^{\star}>t\}),
	$$
	for all $t\geq 0$.\\
	$(\mathrm{b})$ If $u \in L^p(\Omega,\mathfrak{m})$ for some $1 \leq p \leq \infty$, then $u^{\sharp} \in L^p([0, \mathfrak{m}(\Omega)], \mathcal{L}^1)$ and $u^{\star} \in L^p([0, r], \mathfrak{m}_{ N})$. The converse implications also hold. In that case, moreover,
	$$
	\|u\|_{L^p(\Omega, \mathfrak{m})}=\|u^{\sharp}\|_{L^p([0, \mathfrak{m}(\Omega)], \mathcal{L}^1)}=\|u^{\star}\|_{L^p([0, r], \mathfrak{m}_{ N})} .
	$$
	$(\mathrm{c})$ If $u \in L^1(\Omega,\mathfrak{m})$ and let $E \subset \Omega$ be Borel, then
	$$
	\int_E u \mathrm{d}\mathfrak{m} \leq \int_0^{\mathfrak{m}(E)} u^{\sharp}(s) \mathrm{d} s.
	$$
	Moreover, if $u$ is non-negative, then equality holds if and only if $(u|_E)^{\sharp}=(u^{\sharp})|_{[0, \mathfrak{m}(E)]}$.\\
	$(\mathrm{d})$ If $\psi: [0, \infty) \rightarrow[0, \infty)$ be a non-increasing function, then $\psi^{\star}=\psi$ for all $x \in [0, \infty) \backslash L$, where $L$ is a countable set.\\
	$(\mathrm{e})$ $u^\sharp(\mu(t))\geq t$ for any $t\in [0,\operatorname{ess-sup}|u|]\cap [0,\infty)$ and $\mu(u^\sharp(s))\leq s$ for any $s\in [0,\mathfrak{m}(\Omega)]$.
\end{proposition}
\begin{proof}
	(a) and (b) can be proved by standard arguments as in the Euclidean case. For parts (c) and (d), we refer to \cite{MV21}. 
	
	To prove (e), fix $t\geq 0$. If $t\in [0,\operatorname{ess-sup}|u|)$, then $\mu(t)>0$. Let $t^\prime>u^\sharp(\mu(t))$. Hence, we have that $\mu(t^\prime)<\mu(t)$ by the definition of $u^\sharp$. Then, we must have $t^\prime>t$. By letting $t^\prime\rightarrow u^\sharp(\mu(t))$, we get that $u^\sharp(\mu(t))\geq t$. If instead $t= \operatorname{ess-sup}|u|<\infty$, then $\mu(t)=0$ and thus $u^\sharp(\mu(t))=\operatorname{ess-sup}|u|=t$, as desired.
	
	Next, fix $s\in [0,\mathfrak{m}(\Omega)]$. If $s=0$, then $\mu(u^\sharp(0))=0$. If instead $s>0$, then for any $t^\prime>u^\sharp(s)$, we have that $\mu(t^\prime)<s$. Since $\mu$ is right-continuous, by letting $t^\prime\rightarrow u^\sharp(s)$, we obtain that $\mu(u^\sharp(s))\leq s$, as desired.
\end{proof}
\section{Poisson problem on the one-dimensional model space}
Let $N\in (1,\infty)$ and let us define the following density function on $[0,\infty)$:
$$
h_N(t):=N\omega_N t^{N-1},\quad \text{ where }\omega_N=\frac{\pi^{N/2}}{\int_0^\infty t^{N/2}\mathrm{e}^{-t}\mathrm{d}t}.
$$
\begin{definition}[Model spaces]
	Let $N\in (1,\infty)$. We define the one-dimensional model space with dimension parameter $N$ as $([0,\infty),\mathrm{d}_{eu},\mathfrak{m}_N)$, where $\mathfrak{m}_N:=h_N \mathcal{L}^1\llcorner_{[0,\infty)}$, $\mathrm{d}_{eu}$ is the restriction to $[0,\infty)$ of the canonical Euclidean distance over the real line and $\mathcal{L}^1$ is the standard $1$-dimensional Lebesgue measure.
\end{definition}
We denote by $H_N:[0,\infty)\rightarrow [0,\infty)$ the cumulative distribution function of $\mathfrak{m}_N$, i.e.
$$
H_N(x):=\mathfrak{m}_N([0,x])=\omega_N x^N.
$$
\begin{proposition}[A characterization of Sobolev functions on $([0,\infty),\mathrm{d}_{eu},\mathfrak{m}_N)$]\label{propchso}
It holds that
$$
W^{1, p}([0,r), \mathfrak{m}_{N})=\big\{v \in L^p([0,r), \mathfrak{m}_{N}): v^\prime \text { exists and belongs to } L^p([0,r), \mathfrak{m}_{N})\big\}.
$$
Moreover, if $v\in W^{1, p}([0,r), \mathfrak{m}_{N})$, then $|\nabla v|=|v^{\prime}|$, $\mathfrak{m}_{N}$-a.e. in $[0,r)$, where $v^\prime$ is the distributional derivative defined by
\begin{equation}\label{distri deri}
	\int_0^r v(t) \psi^{\prime}(t) \mathrm{d} t=-\int_0^r v^{\prime}(t) \psi(t) \mathrm{d} t, \quad \text{ for any } \psi \in \mathrm{C}_{\mathrm{c}}^1((0, r)).
\end{equation}
\end{proposition}

\begin{proof}
Step 1: we prove that if $v \in W^{1, p}([0,r), \mathfrak{m}_{N})$, then $v^\prime$ exists and $|{v}^{\prime}| \leq |\nabla v|$, $\mathfrak{m}_{N}$-a.e. in $[0,r)$. 

Given $\epsilon\in (0,r/2)$, let us choose a function $\eta \in \mathrm{C}_{\mathrm{c}}^1((0, r))\subset \mathrm{Lip}_{\mathrm{c}}([0, r))$ such that $0\leq \eta\leq 1$ and $\eta=1$ on $[\epsilon,r-\epsilon]$. By the definition of $ W^{1, p}([0,r), \mathfrak{m}_{N})$, we know that $\eta v\in W^{1, p}([0,\infty), \mathfrak{m}_{N})$. Hence, there exists $\{v_n\}_n\subset \operatorname{Lip}_{\mathrm{bs}}([0,\infty))$ such that
	$$
	\|v_n-\eta v\|_{{L}^p}\rightarrow 0,\quad\|\mathrm{lip}(v_n)-|D(\eta v)|_{*,p}\|_{{L}^p}\rightarrow 0.
	$$
	Up to a not relabeled subsequence, we have that $v_n \rightarrow \eta v$, $\mathfrak{m}_N$-a.e. in $[0,\infty)$. Let us denote by $F$ the ($\mathcal{L}^1$-negligible) set of $t \in [\epsilon,r-\epsilon]$ such that either $\lim_n v_n(t)$ does not exist, or $\lim_n v_n(t)$ exists but are different from $\eta v$. Hence, for any $s,t\in [\epsilon,r-\epsilon]\backslash F$ with $s<t$, we have that
	\begin{align*}
		|v(t)-v(s)|&=|(\eta v)(t)-(\eta v)(s)|=\lim_n |v_n(t)-v_n(s)|\\
		&\leq (t-s) \lim_n \int_0^1 \mathrm{lip}(v_n)((1-u)s+ut)\mathrm{d}u\\
		&=\lim_n \int_s^t \mathrm{lip}(v_n)(w)\mathrm{d}w=\int_{s}^t |D(\eta v)|_{*,p}(w)\mathrm{d}w=\int_{s}^t |\nabla v|(w)\mathrm{d}w,
	\end{align*}
	where we used the facts that $\|\mathrm{lip}(v_n)-|D(\eta v)|_{*,p}\|_{{L}^p(\mathfrak{m}_{N})}\rightarrow 0$ and $h_N$ has a positive lower bound on $[\epsilon,r-\epsilon]$. On the other hand, it is easy to see that $|\nabla v|\in L^1([\epsilon,r-\epsilon],\mathcal{L}^1)$. Therefore, $v$ coincides a.e. in $[\epsilon,r-\epsilon]$ with an absolutely continuous map $\hat{v}:[\epsilon,r-\epsilon]\rightarrow \mathbb{R}$ and 
	$$
	|\hat{v}^{\prime}| \leq |\nabla v|,\quad \mathcal{L}^1\text{-a.e. in }(\epsilon,r-\epsilon).
	$$
	In particular, the distributional derivative on $(\epsilon,r-\epsilon)$ of $v$ exists and $|{v}^{\prime}| \leq |\nabla v|$ $\mathcal{L}^1$-a.e. in $(\epsilon,r-\epsilon)$. By the arbitrariness of $\epsilon>0$, we have that the distributional derivative on $(0,r)$ of $v$ exists and $|{v}^{\prime}| \leq |\nabla v|$ $\mathfrak{m}_{N}$-a.e. in $[0,r)$. 
	
Step 2: If $v \in L^p([0,r), \mathfrak{m}_{N})$ is such that $v^\prime$ exists and $v^{\prime} \in L^p([0,r), \mathfrak{m}_{N})$, then $v\in W^{1, p}([0,r), \mathfrak{m}_{N})$ and $|\nabla v|  \leq |{v}^{\prime}|$, $\mathfrak{m}_{N}$-a.e. in $[0,r)$. 

Given any $\epsilon\in (0,r)$, we know that $v,v^\prime \in L^p((\epsilon,r), \mathcal{L}^1\llcorner_{[0,\infty)})$. Now, we fix a sequence $\{\epsilon_n\}_n$ such that $\epsilon_n\searrow 0$. By using the extension theorem of Sobolev functions on $\mathbb{R}$, we can find a sequence $\{v_n\}_n$ of smooth functions on $\mathbb{R}$ such that 
	\begin{align}
		&\|v_n-v\|_{L^p((\epsilon_n,r), \mathcal{L}^1\llcorner_{[0,\infty)})}\rightarrow 0,\text{ and }\|v_n^\prime -v^\prime \|_{L^p((\epsilon_n,r), \mathcal{L}^1\llcorner_{[0,\infty)})}\rightarrow 0,\notag\\
		&\int_0^{\epsilon_n}|v_n|^p\mathrm{d}\mathcal{L}^1\rightarrow 0,\text{ and }\int_0^{\epsilon_n}|v_n^\prime|^p\mathrm{d}\mathcal{L}^1\rightarrow 0.\label{esti}
	\end{align}
	
	Fix $\eta\in \mathrm{Lip}_{\mathrm{c}}([0, r))$ and set $g:=\eta v$. Notice that $\eta v_n\in \mathrm{Lip}_{\mathrm{c}}([0, \infty))$ and $(\eta v_n)^\prime=\mathrm{lip}(\eta v_n)$ $\mathfrak{m}_{N}$-a.e. in $[0, \infty)$ for any $n$. Hence, by using $(\ref{esti})$, it is easy to check that
	\begin{align*}
		\|\eta v_n-g\|_{L^p([0,\infty), \mathfrak{m}_{N})}\rightarrow 0\text{ and }
		\|\mathrm{lip}(\eta v_n)-|\eta v^\prime+\eta^\prime v|\|_{L^p([0,\infty), \mathfrak{m}_{N})}\rightarrow 0.
	\end{align*}
	This implies that $g\in W^{1,p}([0,\infty), \mathfrak{m}_{N})$ and $|\nabla v|=|Dg|_{*,p}\leq |v^\prime|$, $\mathfrak{m}_{N}$-a.e. on $\{\eta=1\}$. By the arbitrariness of $\eta$, we obtain that $v\in W^{1, p}([0,r), \mathfrak{m}_{N})$ and $|\nabla v|\leq |v^\prime|$ $\mathfrak{m}_{N}$-a.e. in $[0,r)$. 
\end{proof}

By Propositions $\ref{prop1}$ and $\ref{propchso}$, we immediately obtain the following corollary:
\begin{corollary}
	Let $p \in(1, \infty)$ and $r>0$. \\
	A function $v\in W_0^{1,p}([0,r),\mathfrak{m}_N)$ is a solution to the Poisson problem 
	\begin{equation}\label{eq2}
		\begin{cases}
			-\mathscr{L}_p v=f, &\text{ in }[0,r),\\
			v=0, &\text{ on }\partial [0,r)=\{r\},
		\end{cases}
	\end{equation}
	if and only if for any $\psi\in W_0^{1,p}([0,r),\mathfrak{m}_N)$, it holds that
	$$
	\int_0^r  v^\prime  \psi^\prime |v^\prime|^{p-2} \mathrm{d}\mathfrak{m}_N=\int_0^r f \psi \mathrm{d}\mathfrak{m}_N.
	$$
\end{corollary}

\begin{proposition}[An explicit solution to Poisson problem on $([0,\infty),\mathrm{d}_{eu},\mathfrak{m}_N)$]\label{propex}
	Let $p \in(1, \infty)$ and $r>0$. Let $f\in L^q([0,r),\mathfrak{m}_N)$ be a non-negative function. Then there exists a unique $($in the $\mathfrak{m}_N$-a.e. sense$)$ solution to the Poisson problem $(\ref{eq2})$, which can be represented as
	\begin{align}
		&v(\rho)=\int_\rho^{r}\bigg(\frac{1}{h_{ N}(u)} \int_0^u f \mathrm{d}\mathfrak{m}_{ N}\bigg)^{\frac{1}{p-1}} \mathrm{d} u, \quad \text{ for any } \rho \in[0, r],\label{ex-so-1}
	\end{align}
	or equivalently as
	\begin{align}
		&v(\rho)=\int_{H_{N}(\rho)}^{H_{ N}(r)} \frac{1}{\mathcal{I}_{ N}(\sigma)}\bigg(\frac{1}{\mathcal{I}_{ N}(\sigma)} \int_0^\sigma f \circ H_{ N}^{-1}(t) \mathrm{d} t\bigg)^{\frac{1}{p-1}} \mathrm{d} \sigma, \quad \text{ for any } \rho \in[0, r],\label{ex-so-2}
	\end{align}
	where $\mathcal{I}_{N}:=h_N\circ H_N^{-1}$.
\end{proposition}
\begin{proof}
	Step 1: we show that $v(\rho)<\infty$ for any $\rho\in (0,r]$. If $\rho=r$, then $v(r)=0$. Now, consider $\rho\in (0,r)$. By H\"{o}lder inequality, we have that
	\begin{equation}\label{Holder}
		\int_0^s|f| \mathrm{d}\mathfrak{m}_{N} \leq\|f\|_{L^q([0,r), \mathfrak{m}_{N})} H_{N}(s)^{\frac{1}{p}},\quad \text { for any } s \in(0, r].
	\end{equation}
	Hence, 
	\begin{align}
		|v(\rho)| & \leq \|f\|_{L^q(\mathfrak{m}_N)}^{\frac{1}{p-1}} \int_\rho^{r} \frac{H_{ N}(u)^{\frac{1}{p(p-1)}}}{h_{ N}(u)^{\frac{1}{p-1}}} \mathrm{d} u \leq C_1\|f\|_{L^q(\mathfrak{m}_N)}^{\frac{1}{p-1}} \int_\rho^{r} \frac{u^{\frac{N}{p(p-1)}}}{u^{\frac{N-1}{p-1}}} \mathrm{d} u \label{w-esti}\\
		& = \begin{cases}
			C_2\|f\|_{L^q(\mathfrak{m}_N)}^{\frac{1}{p-1}}\bigg(r^{\frac{N}{p(p-1)}-\frac{N-1}{p-1}+1}-\rho^{\frac{N}{p(p-1)}-\frac{N-1}{p-1}+1}\bigg), & \frac{N}{p(p-1)}-\frac{N-1}{p-1}>-1, \notag\\
			C_1\|f\|_{L^q(\mathfrak{m}_N)}^{\frac{1}{p-1}}(\log r-\log \rho), & \frac{N}{p(p-1)}-\frac{N-1}{p-1}=-1, \notag\\
			C_2\|f\|_{L^q(\mathfrak{m}_N)}^{\frac{1}{p-1}}\bigg(\rho^{\frac{N}{p(p-1)}-\frac{N-1}{p-1}+1}-r^{\frac{N}{p(p-1)}-\frac{N-1}{p-1}+1}\bigg), & \frac{N}{p(p-1)}-\frac{N-1}{p-1}<-1,
		\end{cases}
	\end{align}
	where $C_1,C_2>0$ are constants depending only on $p$ and $N$.
	
	Step 2: we show that the two expressions $(\ref{ex-so-1})$, $(\ref{ex-so-2})$ are actually equivalent. Indeed, by using the fact $\mathcal{I}_{ N}=h_{ N} \circ H_{ N}^{-1}$, we deduce that
	\begin{align*}
		\int_\rho^{r}\bigg(\frac{1}{h_{ N}(u)} \int_0^u f \mathrm{d}\mathfrak{m}_{ N}\bigg)^{\frac{1}{p-1}} \mathrm{d} u & =\int_\rho^{r} \frac{1}{h_{ N}(u)}\bigg(\frac{1}{h_{ N}(u)} \int_0^u f \mathrm{d}\mathfrak{m}_{ N}\bigg)^{\frac{1}{p-1}} h_{ N}(u) \mathrm{d} u \\
		& =\int_{H_{N}(\rho)}^{H_{N}(r)} \frac{1}{\mathcal{I}_{ N}(\sigma)}\bigg(\frac{1}{\mathcal{I}_{ N}(\sigma)} \int_0^{H_{ N}^{-1}(\sigma)} f(s) \mathrm{d} \mathfrak{m}_{ N}(s)\bigg)^{\frac{1}{p-1}} \mathrm{d} \sigma \\
		& =\int_{H_{N}(\rho)}^{H_{ N}(r)} \frac{1}{\mathcal{I}_{N}(\sigma)}\bigg(\frac{1}{\mathcal{I}_{ N}(\sigma)} \int_0^\sigma f \circ H_{ N}^{-1}(t) \mathrm{d} t\bigg)^{\frac{1}{p-1}} \mathrm{d} \sigma,
	\end{align*}
	where we have used the change of variables $\sigma=H_{N}(u)$ in the external integral, and then the change of variables $t=H_{N}(s)$ in the internal integral.
	
	Step 3: we show that a solution to the Poisson problem $(\ref{eq2})$ must coincide (in the $\mathfrak{m}_N$-a.e. sense) with the function in $(\ref{ex-so-1})$. In particular, we get the uniqueness of solutions to $(\ref{eq2})$. Let $\bar{v} \in W^{1, p}([0,r), \mathfrak{m}_{N})$ be a solution to $(\ref{eq2})$. First, we prove that the distributional derivative as in $(\ref{distri deri})$ of $\bar{v}$ coincides $\mathfrak{m}_{N}$-a.e. in $(0,r)$ with the function $-g^{\frac{1}{p-1}}$, where
	$$
	g(x)=\frac{1}{h_{N}(x)} \int_0^x f \mathrm{d}\mathfrak{m}_{ N}, \quad x \in[0, r].
	$$
	Indeed, for any $\psi \in \mathrm{C}_{\mathrm{c}}^1([0, r))$ one has $(x, s) \mapsto \chi_{[0, x]}(s) f(s) \psi^{\prime}(x) h_{N}(s) \in L^1([0, r]^2,(\mathcal{L}^1)^2)$, and thus by the Fubini Theorem:
	\begin{align}
		\int_0^r g(x) \psi^{\prime}(x) \mathrm{d}\mathfrak{m}_{N}(x) & =\int_0^{r}\bigg(\int_0^{r} \chi_{[0, x]}(s) f(s) \frac{\psi^{\prime}(x)}{h_{N}(x)} \mathrm{d}\mathfrak{m}_{N}(s)\bigg) \mathrm{d}\mathfrak{m}_{N}(x)\notag \\
		& =\int_0^{r} f(s)\bigg(\int_s^{r} \psi^{\prime}(x) \mathrm{d} x\bigg) \mathrm{d}\mathfrak{m}_{N}(s)\notag \\
		& =-\int_0^{r} f \psi \mathrm{d}\mathfrak{m}_{N}.\label{so}
	\end{align}
	Thus, since $\bar{v}$ is a solution to $(\ref{eq2})$, for any $\psi \in \mathrm{C}_{\mathrm{c}}^1([0, r))$ we have that
	$$
	\int_0^r\big[g(x)+|\bar{v}^{\prime}|^{p-2}(x) \bar{v}^{\prime}(x)\big] h_{N}(x) \psi^{\prime}(x) \mathrm{d} x=0.
	$$
	By a classical result, there exists a constant $C \in \mathbb{R}$ such that $|\bar{v}^{\prime}|^{p-2}(x) \bar{v}^{\prime}(x) h_{N}(x)+g(x) h_{N}(x)= C$ for $\mathcal{L}^1$-a.e. $x \in (0,r)$. This however implies that for any $\psi \in \mathrm{C}_{\mathrm{c}}^1([0, r))$,
	$$
	0=C \int_0^{r} \psi^{\prime}(x) \mathrm{d} x=-C \psi(0),
	$$
	hence $C=0$. Moreover, we have $\bar{v}^{\prime}=-g^{\frac{1}{p-1}}$ (recall that $g \geq 0$). Since $\bar{v} \in W_0^{1, p}([0,r), \mathfrak{m}_{N})$, there exists $\{v_n\}_n\subset \operatorname{Lip}_{\mathrm{c}}([0,r))$ such that
	$$
	\|v_n-\bar{v}\|_{{L}^p([0,\infty),\mathfrak{m}_{N})}\rightarrow 0,\quad\||D(v_n-\bar{v})|_{*,p}\|_{{L}^p([0,\infty),\mathfrak{m}_{N})}\rightarrow 0.
	$$
	In particular, $\|(v_n-\bar{v})^\prime\|_{{L}^p([0,\infty),\mathfrak{m}_{N})}\rightarrow 0$ by Proposition $\ref{propchso}$ and up to a not relabeled subsequence, we have that $v_n \rightarrow \bar{v}$ $\mathfrak{m}_N$-a.e. in $[0,\infty)$. Fix $\epsilon\in(0,r)$. Let us denote by $F$ the ($\mathcal{L}^1$-negligible) set of $t \in [\epsilon,r]$ such that either $\lim_n v_n(t)$ does not exist, or $\lim_n v_n(t)$ exists but are different from $\bar{v}$. Hence, for any $s\in [\epsilon,r]\backslash F$, we have that
	\begin{align*}
		-\bar{v}(s)=\lim_n [v_n(r)-v_n(s)]=\lim_n \int_s^r v_n^\prime (u)\mathrm{d}u=\int_s^r \bar{v}^\prime (u)\mathrm{d}u,
	\end{align*}
	where we used the facts that $\|(v_n-\bar{v})^\prime\|_{{L}^p([0,\infty),\mathfrak{m}_{N})}\rightarrow 0$ and $h_N$ has a positive lower bound on $[\epsilon,r]$. By the arbitrariness of $\epsilon>0$, we have that
	$$
	\bar{v}(\rho)=-\int_\rho^{r} \bar{v}^{\prime}(s) \mathrm{d} s=\int_\rho^{r} g^{\frac{1}{p-1}}(s) \mathrm{d} s=v(\rho),\quad \mathfrak{m}_{ N} \text{-a.e. } \rho\in (0,r).
	$$
	
	Step 4: we show that $v$ defined by $(\ref{ex-so-1})$ is actually a solution to $(\ref{eq2})$. Observe that $v$ is $\mathrm{C}^1$ on $(0, r]$ and continuous on $[0, r]$ (with $v(r)=0$) since the integrand is continuous on $(0, r]$. By straightforward computations, we show that $v$ and $v^{\prime}$ belong to $L^p([0,r), \mathfrak{m}_{N})$ and thus $v\in W^{1, p}([0,r), \mathfrak{m}_{N})$ by Proposition $\ref{propchso}$. By $(\ref{w-esti})$, if $\frac{N}{p(p-1)}-\frac{N-1}{p-1} \neq-1$, then
	\begin{equation}\label{2.21}
		\int_0^{r}|v|^p \mathrm{d}\mathfrak{m}_{N} \leq C_3\|f\|_{L^q(\mathfrak{m}_{N})}^{\frac{p}{p-1}} \int_0^{r}\big(r^{\frac{N}{p-1}-\frac{(N-1) p}{p-1}+p}+\rho^{\frac{N}{p-1}-\frac{(N-1) p}{p-1}+p}\big) \rho^{N-1} \mathrm{d} \rho<\infty;
	\end{equation}
	if instead $\frac{N}{p(p-1)}-\frac{N-1}{p-1}=-1$, then
	\begin{equation}\label{2.22}
		\int_0^{r}|v|^p \mathrm{d}\mathfrak{m}_{N} \leq C_3\|f\|_{L^q(\mathfrak{m}_{N})}^{\frac{p}{p-1}} \int_0^{r}\big(\log r-\log \rho\big)^p \rho^{N-1} \mathrm{d} \rho<\infty.
	\end{equation}
	Moreover, exploiting $(\ref{Holder})$, we have that
	\begin{equation}\label{2.23}
		\int_0^{r}|v^{\prime}|^p \mathrm{d}\mathfrak{m}_{N} \leq\|f\|_{L^q(\mathfrak{m}_{N})}^{\frac{p}{p-1}} \int_0^{r} \frac{H_{ N}(s)^{\frac{1}{p-1}}}{h_{ N}(s)^{\frac{1}{p-1}}} \mathrm{d} s \leq C_4\|f\|_{L^q(\mathfrak{m}_{N})}^{\frac{p}{p-1}} \int_0^{r} \frac{s^{\frac{N}{p-1}}}{s^{\frac{N-1}{p-1}}} \mathrm{d} r<\infty.
	\end{equation}
	We remark that all constants $C_1, C_2, C_3$ and $C_4$ depend only on $N$ and $p$. 
	
	Extend $v$ to $[0,\infty)$ by taking it to be $0$ on $(r,\infty)$. Next, we prove that $v \in W_0^{1, p}([0,r), \mathfrak{m}_{N})$. Next let $\zeta:\mathbb{R}\rightarrow [0,1]$ be a $1$-Lipschitz function such that
	$$
	\zeta = 1 \text { on }(-\infty,1],\quad \zeta = 0 \text { on } [2,\infty),
	$$
	and write
	$$
	\left\{\begin{array}{l}
		\zeta_m(x):=\zeta(m (r-x)), \quad x \in [0,\infty), \\
		w_m:=v(1-\zeta_m) .
	\end{array}\right.
	$$
	Then $w_{m}^\prime(x)=v^\prime(x)(1-\zeta_m(x))+m v \zeta^{\prime}(m (r-x))$ for a.e. $x\in (0,r)$.
	Consequently,
	\begin{align*}
		\int_0^\infty|w_{m}^\prime-v^\prime|^p \mathrm{d}\mathfrak{m}_{N}&=\int_0^r|v^\prime(x)(1-\zeta_m(x))+m v \zeta^{\prime}(m (r-x))-v^\prime|^p \mathrm{d}\mathfrak{m}_{N}\\ 
		&\leq  2^{p-1} \int_0^r|\zeta_m|^p|v^\prime|^p\mathrm{d}\mathfrak{m}_{N} 
		+ 2^{p-1}m^p \int_{r-2/m}^{r}|v|^p \mathrm{d}\mathfrak{m}_{N} \\
		&:=  A+B.
	\end{align*}
	Now
	$$
	A \rightarrow 0 \quad \text { as } m \rightarrow \infty
	$$
	since $\zeta_m(x)\neq 0$ only if $r-2/m \leq x \leq r$. To estimate the term $B$:
	\begin{align*}
		B &=2^{p-1}m^p \int_{r-2/m}^{r}|v(x)-v(r)|^p \mathrm{d}\mathfrak{m}_{N}(x)\\
		& \leq 2^{p-1} m^ph_N(r)\int_{r-2/m}^{r} 
		(r-x)^{p-1} \mathrm{d}x\int_{r-2/m}^{r} |v^\prime(s)|^p \mathrm{d} s \\
		& = 2^{2p-1} p^{-1}h_N(r)  \int_{r-2/m}^{r} |v^\prime(s)|^p \mathrm{d} s  \rightarrow 0 \quad \text { as } m \rightarrow 0.
	\end{align*}
	Since clearly $w_m \rightarrow v$ in $L^p([0,\infty), \mathfrak{m}_{N})$, we conclude
	$$
	w_m \rightarrow v \quad \text { in } W^{1, p}([0,\infty), \mathfrak{m}_{N}).
	$$
	But $w_m(x)=0$ if $r-1/m\leq x\leq r$. We can therefore approximate the $w_m$ to produce functions $v_m \in \mathrm{Lip}_{\mathrm{c}}([0,r))$ such that $v_m \rightarrow v$ in $ W^{1, p}([0,\infty), \mathfrak{m}_{N})$. Hence $v \in W_0^{1, p}([0,r), \mathfrak{m}_{N})$.
	
	Finally, by tracing back the identity in $(\ref{so})$, the very same argument shows that $v$ is a solution to $-\mathscr{L}_p v=f$.
\end{proof}

\section{A Talenti-type comparison theorem for $\mathrm{RCD}(0,N)$ spaces}
Throughout this section, let $(\mathrm{X}, \mathrm{d}, \mathfrak{m})$ be a $\operatorname{RCD}(0, N)$ space for some $N \in(1, \infty)$ and let $\Omega$ be an open subset of $\mathrm{X}$ with $\mathfrak{m}(\Omega)\in (0,\infty)$. Let $p, q \in(1, \infty)$ be conjugate exponents.

Before passing to the proof of the main comparison theorem, we state two Lemmas. For their proof, we refer to \cite{MV21} and \cite{Wu24}.
\begin{lemma}\label{Talem0}
	Let $u \in L^p(\Omega,\mathfrak{m})$, $f \in L^q(\Omega,\mathfrak{m})$. Define
	$$
	F(t)=\int_{\{u>t\}}(u-t) f \mathrm{d}\mathfrak{m}, \quad \text{ for any } t \in \mathbb{R}.
	$$
	Then $F$ is differentiable out of a countable set $C \subset \mathbb{R}$, and
	$$
	F^{\prime}(t)=-\int_{\{u>t\}} f\mathrm{d}\mathfrak{m}, \quad \text{ for any } t \in \mathbb{R} \backslash C.
	$$
\end{lemma}

\begin{lemma}\label{Talem1}
	Let $f \in L^q(\Omega,\mathfrak{m})$. Let $u \in W_0^{1, p}(\Omega,\mathfrak{m})$ be a solution to the Poisson problem $(\ref{eq1})$. Then for $\mathcal{L}^1$-a.e. $t>0$, it holds that
	$$
	-\frac{\mathrm{d}}{\mathrm{d} t} \int_{\{|u|>t\}}|\nabla u| \mathrm{d}\mathfrak{m} \leq\big(-\mu^{\prime}(t)\big)^{1 / q}\bigg(\int_{\{|u|>t\}}|f| \mathrm{d}\mathfrak{m}\bigg)^{1 / p},
	$$
	and 
	$$
	-\frac{\mathrm{d}}{\mathrm{d} t} \int_{\{|u|>t\}}|\nabla u|^p \mathrm{d}\mathfrak{m} \leq\int_{\{|u|>t\}}|f| \mathrm{d}\mathfrak{m},
	$$
	where $\mu$ is the distribution function of $u$.
\end{lemma} 

We can now prove the Talenti-type comparison theorem. A similar result was proved in \cite{MV21} in the case of positive curvature.
\begin{theorem}[A Talenti-Type Comparison Theorem]
	Let $p, q \in(1, \infty)$ be conjugate exponents. Let $(\mathrm{X}, \mathrm{d}, \mathfrak{m})$ be a $\mathrm{RCD}(0, N)$ space with $N \in(1, \infty)$, and $\Omega \subset \mathrm{X}$ an open set with $\mathfrak{m}(\Omega)=a \in(0,\infty)$.
	
	Let $f \in L^q(\Omega, \mathfrak{m})$. Assume that $u \in W_0^{1, p}(\Omega, \mathfrak{m})$ is a solution to the Poisson problem $(\ref{eq1})$. Let also $v \in W^{1, p}([0, r_a), \mathfrak{m}_{N})$ be a solution to the problem
	\begin{equation*}
		\begin{cases}
			-\mathscr{L}_p v=f^\star, &\text{ in }[0,r_a),\\
			v=0, &\text{ on }\partial [0,r_a)=\{r_a\},
		\end{cases}
	\end{equation*}
	where $r_a>0$ is such that $\mathfrak{m}_N([0,r_a])=\mathfrak{m}(\Omega)$, and $f^\star$ is the Schwarz symmetrization of $f$. Then\\
	$(\mathrm{i})$ $\mathrm{AVR}_{\mathfrak{m}}^{q/N}u^{\star} \leq v$, $\mathfrak{m}_{N}$-a.e. in $[0, r_a]$. In addition, this inequality is sharp.\\
	$(\mathrm{ii})$ For any $1 \leq r \leq p$, the following $L^r$-gradient estimate holds:
	\begin{equation}\label{r-grad}
		\mathrm{AVR}_{\mathfrak{m}}^{\frac{r}{(p-1)N}}\int_{\Omega}|\nabla u|^r \mathrm{d}\mathfrak{m} \leq \int_0^{r_a}|\nabla v|^r \mathrm{d}\mathfrak{m}_{N}.
	\end{equation}
\end{theorem}
\begin{proof}
	Without loss of generality, we can assume that $u \neq 0$.\\
	(i) Combining Theorem $\ref{thmiso}$, coarea formula $(\ref{coeq4})$, Chain rule of minimal $p$-relaxed slope, Lemma $\ref{Talem1}$ and Proposition $\ref{propre}$(c), we obtain the following chain of inequalities:
	\begin{align}
		\mathrm{AVR}_{\mathfrak{m}}^{\frac{1}{N}}\mathcal{I}_{N}(\mu(t)) & \leq \operatorname{Per}(\{|u|>t\})=-\frac{\mathrm{d}}{\mathrm{d} t} \int_{\{|u|>t\}}|\nabla| u| | \mathrm{d}\mathfrak{m}=-\frac{\mathrm{d}}{\mathrm{d} t} \int_{\{|u|>t\}}|\nabla u| \mathrm{d}\mathfrak{m}\notag \\
		& \leq\big(-\mu^{\prime}(t)\big)^{1 / q}\bigg(\int_{\{|u|>t\}}|f| \mathrm{d}\mathfrak{m}\bigg)^{1 / p} \leq\big(-\mu^{\prime}(t)\big)^{1 / q}\bigg(\int_0^{\mu(t)} f^{\sharp}(s) \mathrm{d} s\bigg)^{1 / p},\label{thmeq1.0}
	\end{align}
	for $\mathcal{L}^1$-a.e. $t>0$, which can be rewritten as
	\begin{equation}\label{thmeq1.1}
		1 \leq\frac{-\mu^{\prime}(t)}{\mathrm{AVR}_{\mathfrak{m}}^{\frac{q}{N}}\mathcal{I}_{ N}(\mu(t))}\bigg(\frac{1}{\mathcal{I}_{N}(\mu(t))} \int_0^{\mu(t)} f^{\sharp}(s) \mathrm{d} s\bigg)^{\frac{1}{p-1}},
	\end{equation}
	for $\mathcal{L}^1$-a.e. $t \in(0, M)$, where $M=\operatorname{ess-sup} |u| \in(0, \infty]$. For simplicity, let
	$$
	F(\xi)=\int_0^{\xi} f^{\sharp}(s) \mathrm{d} s, \quad \text { for any } \xi>0.
	$$
	Let now $0 \leq \tau^{\prime}<\tau\leq M$. Integrating $(\ref{thmeq1.1})$ from $\tau^{\prime}$ to $\tau$ we get
	$$
	\tau-\tau^{\prime} \leq \int_{\tau^{\prime}}^\tau \frac{-\mu^{\prime}(t)}{\mathrm{AVR}_{\mathfrak{m}}^{\frac{q}{N}}\mathcal{I}_{ N}(\mu(t))}\bigg(\frac{F(\mu(t))}{\mathcal{I}_{ N}(\mu(t))}\bigg)^{\frac{1}{p-1}} \mathrm{d} t, \quad 0 \leq \tau^{\prime}<\tau\leq M .
	$$
	Using the change of variables $\xi=\mu(t)$ on the intervals where $\mu$ is continuous(see \cite[p. 108]{Fo99}), and observing that the integrand is non-negative, we obtain
	\begin{equation}\label{thmeq1.2}
		\tau-\tau^{\prime} \leq \int_{\mu(\tau)}^{\mu(\tau^{\prime})} \frac{1}{\mathrm{AVR}_{\mathfrak{m}}^{\frac{q}{N}}\mathcal{I}_{ N}(\xi)}\bigg(\frac{F(\xi)}{\mathcal{I}_{ N}(\xi)}\bigg)^{\frac{1}{p-1}} \mathrm{d} \xi, \quad 0 \leq \tau^{\prime}<\tau\leq M .
	\end{equation}
	Let us fix $s \in(0, \mu(0))$ with $u^{\sharp}(s)>0$ and let $\eta>0$ be a small enough parameter such that $u^{\sharp}(s)>\eta$; consider $\tau^{\prime}=0$ and $\tau=u^{\sharp}(s)-\eta$. Notice that, since $u^{\sharp}(s)$ is the infimum of the $\tilde{\tau}$ such that $\mu(\tilde{\tau})<s$, we have that $\mu(\tau) \geq s$. Using $(\ref{thmeq1.2})$ and the non-negativity of the integrand again, we obtain that
	$$
	u^{\sharp}(s)-\eta \leq \int_{s}^{\mu(0)} \frac{1}{\mathrm{AVR}_{\mathfrak{m}}^{\frac{q}{N}}\mathcal{I}_{ N}(\xi)}\bigg(\frac{F(\xi)}{\mathcal{I}_{ N}(\xi)}\bigg)^{\frac{1}{p-1}} \mathrm{d} \xi, \quad \text{ for every } s \in(0, \mu(0)) .
	$$
	Letting $\eta \downarrow 0$, enlarging the integration interval and notice that the function $u^{\sharp}$ vanishes on $(\mu(0), \mathfrak{m}(\Omega)]$ and $u^{\sharp}$ is left continuous on $[0, \mathfrak{m}(\Omega)]$ and right-continuous at $0$, we get that
	$$
	u^{\sharp}(s) \leq \int_s^{\mathfrak{m}(\Omega)} \frac{1}{\mathrm{AVR}_{\mathfrak{m}}^{\frac{q}{N}}\mathcal{I}_{ N}(\xi)}\bigg(\frac{1}{\mathcal{I}_{ N}(\xi)} \int_0^{\xi} f^{\sharp}(t) \mathrm{d} t\bigg)^{\frac{1}{p-1}} \mathrm{d} \xi, \quad \text{ for every }  s \in[0, \mathfrak{m}(\Omega)] .
	$$
	Finally, by the definition of the Schwarz symmetrization, we obtain
	\begin{align*}
		u^{\star}(x) &\leq \int_{H_{N}(x)}^{\mathfrak{m}(\Omega)} \frac{1}{\mathrm{AVR}_{\mathfrak{m}}^{\frac{q}{N}}\mathcal{I}_{ N}(\xi)}\bigg(\frac{1}{\mathcal{I}_{ N}(\xi)} \int_0^{\xi} f^{\star} \circ H_{N}^{-1}(t) \mathrm{d} t\bigg)^{\frac{1}{p-1}} \mathrm{d} \xi,\\
		&=\frac{1}{\mathrm{AVR}_{\mathfrak{m}}^{\frac{q}{N}}}v(x), \quad \text{ for every }  x \in[0, r_a],
	\end{align*}
	where we used the characterization of $v$ we obtained in $(\ref{ex-so-2})$ and $H_{ N}(r_a)=\mathfrak{m}(\Omega)$.
	
	Finally, if $(\mathrm{X}, \mathrm{d}, \mathfrak{m})=([0,\infty),\mathrm{d}_{eu},\mathfrak{m}_N)$, and assume that $f: [0,\infty) \rightarrow[0, \infty)$ is a non-increasing and non-negative function, then, by Proposition $\ref{propre}$(d), we obtain $f^{\star}=f$ $\mathfrak{m}_{N}$-a.e. and $v^{\star}=v$ on $[0,\infty)$ since $v$ is continuous. On the other hand, we have that $\mathrm{AVR}_{\mathfrak{m}_N}=1$ by a simple computation. Hence, the Talenti-type comparison inequality is sharp. We complete the proof of (i).
	
	(ii) We start by noticing that for $1 \leq r \leq p$, we have
	$$
	\int_{\Omega}|\nabla u|^r \mathrm{d}\mathfrak{m}=\int_{\{|u|>0\}}|\nabla u|^r \mathrm{d}\mathfrak{m}
	$$
	since $|\nabla u|=0$, $\mathfrak{m}$-a.e. in $\{u=0\}$, by the locality of $\nabla$. Let $M:=\operatorname{ess-sup}|u| \in(0,\infty]$; fix $t, h>0$. Using the Hölder inequality, we infer that
	\begin{equation}\label{thmeq1.4}
		\frac{1}{h} \int_{\{t<|u| \leq t+h\}}|\nabla u|^r \mathrm{d}\mathfrak{m} \leq\bigg(\frac{1}{h} \int_{\{t<|u| \leq t+h\}}|\nabla u|^p \mathrm{d}\mathfrak{m}\bigg)^{\frac{r}{p}}\bigg(\frac{\mathfrak{m}(\{t<|u| \leq t+h\})}{h}\bigg)^{\frac{p-r}{p}}.
	\end{equation}
	By Lemma $\ref{Talem1}$, letting $h$ tend to zero in $(\ref{thmeq1.4})$, we obtain
	$$
	-\frac{\mathrm{d}}{\mathrm{d} t} \int_{\{|u|>t\}}|\nabla u|^r \mathrm{d}\mathfrak{m} \leq\bigg(\int_{\{|u|>t\}}|f| \mathrm{d}\mathfrak{m}\bigg)^{\frac{r}{p}}\bigg(-\mu^{\prime}(t)\bigg)^{\frac{p-r}{p}},
	$$
	for $\mathcal{L}^1$-a.e. $t>0$. Let us now adopt again the notation
	$$
	F(\xi)=\int_0^{\xi} f^{\sharp}(s) \mathrm{d} s.
	$$
	Using Propositions $\ref{propre}$(c) again, we get that
	\begin{equation}\label{thmeq1.5}
		-\frac{\mathrm{d}}{\mathrm{d} t} \int_{\{|u|>t\}}|\nabla u|^r \mathrm{d}\mathfrak{m} \leq F(\mu(t))^{\frac{r}{p}}(-\mu^{\prime}(t))^{\frac{p-r}{p}}
	\end{equation}
	for $\mathcal{L}^1$-a.e. $t$. In order to obtain a clean term $\mu^{\prime}(t)$ at the right hand side, we multiply both sides of $(\ref{thmeq1.5})$ with the respective sides of $(\ref{thmeq1.1})$ raised at the power $r / p$. This gives, for $\mathcal{L}^1$-a.e. $t \in(0, M)$:
	$$
	-\frac{\mathrm{d}}{\mathrm{d} t} \int_{\{|u|>t\}}|\nabla u|^r \mathrm{d}\mathfrak{m}  \leq\bigg(\frac{F(\mu(t))}{\mathrm{AVR}_{\mathfrak{m}}^{\frac{1}{N}}\mathcal{I}_{ N}(\mu(t))}\bigg)^{\frac{r}{p-1}}(-\mu^{\prime}(t)).
	$$
	Combining this and coarea formula $(\ref{coeq4})$, and changing the variables as usual with $\xi=\mu(t)$, the following estimate holds that
	\begin{equation}\label{thmeq1.3}
		\int_{\Omega}|\nabla u|^r \mathrm{d} \mathfrak{m}\leq \int_0^{\mathfrak{m}(\Omega)}\bigg(\frac{F(\xi)}{\mathrm{AVR}_{\mathfrak{m}}^{\frac{1}{N}}\mathcal{I}_{ N}(\xi)}\bigg)^{\frac{r}{p-1}} \mathrm{d} \xi .
	\end{equation}
	Finally, we recall that $v$ has an explicit expression we can differentiate. Differentiating $(\ref{ex-so-1})$ (with datum $f^{\star}$), we have, for all $\rho \in(0, r_a)$,
	$$
	v^{\prime}(\rho)=-\bigg(\frac{1}{h_{N}(\rho)} \int_0^{H_{ N}(\rho)} f^{\star}(H_{ N}^{-1}(t)) \mathrm{d} t\bigg)^{\frac{1}{p-1}}=-\bigg(\frac{F(H_{ N}(\rho))}{h_{ N}(\rho)}\bigg)^{\frac{1}{p-1}}.
	$$
	Thus, the following identity holds that
	$$
	\int_0^{r_a}|v^{\prime}|^r \mathrm{d}\mathfrak{m}_{ N}=\int_0^{r_a}\bigg(\frac{F(H_{ N}(\rho))}{h_{ N}(\rho)}\bigg)^{\frac{r}{p-1}} h_{N}(\rho) \mathrm{d} \rho=\int_0^{\mathfrak{m}(\Omega)}\bigg(\frac{F(\xi)}{\mathcal{I}_{ N}(\xi)}\bigg)^{\frac{r}{p-1}} \mathrm{d} \xi,
	$$
	where we have used the change of variables $\xi=H_{ N}(\rho)$ and the fact that $\mathcal{I}_{N}(\xi)=h_{ N}(H_{N}^{-1}(\xi))$. Comparing with $(\ref{thmeq1.3})$, we obtain the desired $L^r$-gradient estimate.
\end{proof}

\section{Rigidity}
Let $(F, \mathrm{d}_F, \mathfrak{m}_F)$ be a length metric measure space. Let $\mathrm{d}_C$ be the pseudo-distance on $[0,\infty) \times F$ defined by
$$
\mathrm{d}_C((p, x),(q, y))=\inf \{L(\gamma) : \gamma(0)=(p, x), \gamma(1)=(q, y)\},
$$
where, for any absolutely continuous curve $\gamma=(\gamma_1, \gamma_2):[0,1] \rightarrow[0,\infty) \times F$,
$$
L(\gamma)=\int_0^1\big(|\dot{\gamma}_1|^2+ \gamma_1^2|\dot{\gamma}_2|^2\big)^{\frac{1}{2}} \mathrm{d} t.
$$

The pseudo-distance $\mathrm{d}_C$ induces an equivalence relation on $[0,\infty) \times F$ by $(p, x) \sim(q, y)$ iff $\mathrm{d}_C((p, x),(q, y))=0$ and then a metric on the quotient. With a common slight abuse of notation we shall denote the completion of such quotient by $([0,\infty) \times_{t}^{N-1} F, \mathrm{d}_C)$. Let us denote by $\pi:[0,\infty) \times F \rightarrow [0,\infty) \times_{t}^{N-1} F$ the quotient map and then define the measure $\mathfrak{m}_C$ on $([0,\infty) \times_{t}^{N-1} F, \mathrm{d}_C)$:
$$
\mathfrak{m}_C=\pi_\sharp \big(t^{N-1} \mathcal{L}^1 \times \mathfrak{m}_F\big).
$$
\begin{definition}[Euclidean metric measure cone]
	We say that a $\mathrm{RCD}(0,N)$ space $(\mathrm{X},\mathrm{d},\mathfrak{m})$ is a Euclidean metric measure cone if it is isomorphic to $([0,\infty) \times_{t}^{N-1} F,\mathrm{d}_C,\mathfrak{m}_C)$ for a $\mathrm{RCD}(N-2,N-1)$ space $(F, \mathrm{d}_F, \mathfrak{m}_F)$. The point $\pi(\{0\}\times F)$ is called tip of the cone.
\end{definition}

\begin{theorem}[Rigidity of the Euclidean isoperimetric inequality]\label{isorig}
	Let $(\mathrm{X}, \mathrm{d}, \mathfrak{m})$ be a $\operatorname{RCD}(0, N)$ space for some $N>1$ with $\mathrm{AVR}_{\mathfrak{m}}>0$. Then the equality $(\ref{iso})$ holds for some Borel set $E \subset \mathrm{X}$ if and only if $\mathrm{X}$ is isometric to a Euclidean metric measure cone over a $\operatorname{RCD}(N-2, N-1)$ space and $E$ is isometric to the ball centered at one of the tips of $\mathrm{X}$. 
\end{theorem}
Theorem $\ref{isorig}$ is stated in \cite{CM24} with the extra assumption that $E$ is bounded, however this assumption can be dropped thanks to the recent \cite{APPV23}.

\begin{proposition}[Euclidean P\'{o}lya-Szeg\"{o} inequality\cite{NV22}]
	Let $(\mathrm{X}, \mathrm{d}, \mathfrak{m})$ be a $\operatorname{CD}(0, N)$ space with $N>1$, and $\Omega \subset \mathrm{X}$ an open set with $\mathfrak{m}(\Omega) \in(0,\infty)$. Then for any $u \in W_0^{1, p}(\Omega, \mathfrak{m})$, it holds that 
	\begin{equation}\label{rigeq1}
		\mathrm{AVR}_{\mathfrak{m}}^{p/N}\int_0^{r_a}|\nabla u^{\star}|^p  \mathrm{d}\mathfrak{m}_{N} \leq \int_{\Omega}|\nabla u|^p  \mathrm{d}\mathfrak{m}.
	\end{equation}
\end{proposition}

\begin{proposition}[Rigidity of the Euclidean P\'{o}lya-Szeg\"{o} inequality\cite{NV24}]
	Let $(\mathrm{X}, \mathrm{d}, \mathfrak{m})$ be a $\operatorname{RCD}(0, N)$ space for some $N>1$ with $\mathrm{AVR}_{\mathfrak{m}}>0$, and let $\Omega \subset \mathrm{X}$ be an open set with $\mathfrak{m}(\Omega) \in(0,\infty)$. 
	
	If there exists $u \in W_0^{1, p}(\Omega, \mathfrak{m}) \cap \operatorname{Lip}(\Omega)$ achieving equality in $(\ref{rigeq1})$, with $|\nabla u| \neq 0$, $\mathfrak{m}$-a.e. in $\{u\neq 0\}$, then $(\mathrm{X}, \mathrm{d}, \mathfrak{m})$ is a Euclidean metric measure cone and $u$ is radial: that is, $u=u^\star(\mathrm{AVR}_{\mathfrak{m}}^{\frac{1}{N}}\mathrm{d}(\cdot, x_0))$, where $x_0$ is a tip of a Euclidean metric measure cone structure of $\mathrm{X}$.
\end{proposition}

The following rigidity result for the Talenti-type comparison theorem will build on top of the rigidity in the Euclidean isoperimetric inequality and Euclidean P\'{o}lya-Szeg\"{o} inequality. A similar result was proved in \cite{MV21} in the case of positive curvature.
\begin{theorem}[Rigidity for Talenti]
	Let $(\mathrm{X}, \mathrm{d}, \mathfrak{m})$ be a $\operatorname{RCD}(0, N)$ space for some $N>1$ with $\mathrm{AVR}_{\mathfrak{m}}>0$, and let $\Omega \subset \mathrm{X}$ be an open set with $\mathfrak{m}(\Omega)=a \in(0,\infty)$.  
	
	Let $f \in L^q(\Omega, \mathfrak{m})$ with $f \neq 0$. Assume that $u \in W_0^{1, p}(\Omega, \mathfrak{m})$ is a solution to the Poisson problem $(\ref{eq1})$. Let also $v \in W_0^{1, p}([0, r_a), \mathfrak{m}_{N})$ be the continuous solution to the problem
	\begin{equation*}
		\begin{cases}
			-\mathscr{L}_p v=f^\star, &\text{ in }[0,r_a),\\
			v=0, &\text{ on }\partial [0,r_a)=\{r_a\},
		\end{cases}
	\end{equation*}
	where $r_a>0$ is such that $\mathfrak{m}_N([0,r_a])=\mathfrak{m}(\Omega)$, and $f^\star$ is the Schwarz symmetrization of $f$. Assume that $\mathrm{AVR}_{\mathfrak{m}}^{q/N}u^{\star}(\bar{x})=v(\bar{x})$ for a point $\bar{x} \in[0, r_a)$. Then\\
	$(\mathrm{i})$ $\mathrm{AVR}_{\mathfrak{m}}^{q/N}u^{\star}=v$ in the whole interval $[\bar{x}, r_a];$\\
	$(\mathrm{ii})$ $(\mathrm{X}, \mathrm{d}, \mathfrak{m})$ is a Euclidean metric measure cone;\\
	$(\mathrm{iii})$ if $\bar{x}=0$, $u \in \operatorname{Lip}(\Omega)$ and $|\nabla u| \neq 0$, $\mathfrak{m}$-a.e. in $\{u\neq 0\}$, then $u$ is radial.
\end{theorem}
\begin{proof}
	To prove (i), it is enough to show that
	$$
	\mathrm{AVR}_{\mathfrak{m}}^{-q/N}v^\sharp(s)=\int_{s}^{\mathfrak{m}(\Omega)} \frac{1}{\mathrm{AVR}_{\mathfrak{m}}^{q/N}\mathcal{I}_{ N}(\sigma)}\bigg(\frac{F(\sigma)}{\mathcal{I}_{ N}(\sigma)} \bigg)^{\frac{1}{p-1}} \mathrm{d} \sigma=u^{\sharp}(s),
	$$ 
	for any $s\in (\bar{s},\mathfrak{m}(\Omega))$, where $\bar{s}=H_N(\bar{x})\in [0,\mathfrak{m}(\Omega))$ and $F(\sigma)=\int_0^\sigma f^\sharp (t) \mathrm{d} t$. By the P\'{o}lya-Szeg\"{o} inequality, $u^\star$ belongs to $W^{1, p}([0, r_a), \mathfrak{m}_{N})$ and it is thus locally absolutely continuous in $(0,r_a)$; the same conclusion thus holds for $u^\sharp$. In fact, $u^\sharp$ is continuous in the whole closed inteval $[0,\mathfrak{m}(\Omega)]$ since it is left-continuous and it is also right-continuous at $0$.
	
	First, we notice that for any $s>\bar{s}$, $u^\sharp (s)<u^\sharp (\bar{s})$. If not, then there exists $s_1>\bar{s}$ such that $u^\sharp (s_1)=u^\sharp (\bar{s})$. Hence, by the Talenti-type comparison theorem, it holds that
	$$
	\mathrm{AVR}_{\mathfrak{m}}^{-q/N}v^\sharp(s_1)\geq u^\sharp (s_1)=u^\sharp (\bar{s})=\mathrm{AVR}_{\mathfrak{m}}^{-q/N}v^\sharp(\bar{s}),
	$$
	which contradicts the strict monotonicity of $v^\sharp$. Let $\{s_n\}_n\subset (\bar{s},\mathfrak{m}(\Omega))$ be such that $s_n\searrow \bar{s}$ and hence
	$$
	\bar{s}\leq \liminf_{n\rightarrow\infty}\mu(u^\sharp(s_n))\leq \limsup_{n\rightarrow\infty}\mu(u^\sharp(s_n))\leq \limsup_{n\rightarrow\infty}s_n=\bar{s},
	$$ 
	by Proposition $\ref{propre}$(e) and $\mu^\sharp (s_n)<\mu^\sharp (\bar{s})$ for all $n$. Then, as in the proof of Talenti-type comparison theorem, we can deduce that
	\begin{align*}
		u^\sharp(\bar{s})&=\lim_{n\rightarrow \infty} u^\sharp(s_n)\leq \limsup_{n\rightarrow \infty}\int_{0}^{u^\sharp(s_n)} \frac{-\mu^{\prime}(t)}{\mathrm{AVR}_{\mathfrak{m}}^{\frac{q}{N}}\mathcal{I}_{ N}(\mu(t))}\bigg(\frac{F(\mu(t))}{\mathcal{I}_{ N}(\mu(t))}\bigg)^{\frac{1}{p-1}} \mathrm{d} t\\
		&\leq  \lim_{n\rightarrow \infty}\int_{\mu(u^\sharp(s_n))}^{\mu(0)} \frac{1}{\mathrm{AVR}_{\mathfrak{m}}^{\frac{q}{N}}\mathcal{I}_{ N}(\sigma)}\bigg(\frac{F(\sigma)}{\mathcal{I}_{ N}(\sigma)}\bigg)^{\frac{1}{p-1}} \mathrm{d} \sigma\\
		&=\int_{\bar{s}}^{\mu(0)} \frac{1}{\mathrm{AVR}_{\mathfrak{m}}^{\frac{q}{N}}\mathcal{I}_{ N}(\sigma)}\bigg(\frac{F(\sigma)}{\mathcal{I}_{ N}(\sigma)}\bigg)^{\frac{1}{p-1}} \mathrm{d} \sigma\stackrel{\text{by assumption}}{\leq} u^\sharp(\bar{s}).
	\end{align*}
	Hence, all the inequalities in the estimate are actually equalities and thus we must have:
	\begin{itemize}
		\item $\mu$ is continuous in $[0,u^\sharp(\bar{s}));$ in particular, $\mu(u^\sharp(s))=s$ for all $s\in (\bar{s},\mathfrak{m}(\Omega)]$.
		\item $\mu(0)=\mathfrak{m}(\Omega);$
		\item for any $s\in (\bar{s},\mathfrak{m}(\Omega))$,
		$$
		\int_{0}^{u^\sharp(s)} \frac{-\mu^{\prime}(t)}{\mathcal{I}_{ N}(\mu(t))}\bigg(\frac{F(\mu(t))}{\mathcal{I}_{ N}(\mu(t))}\bigg)^{\frac{1}{p-1}} \mathrm{d} t=\int_{\mu(u^\sharp(s))}^{\mu(0)} \frac{1}{\mathcal{I}_{ N}(\sigma)}\bigg(\frac{F(\sigma)}{\mathcal{I}_{ N}(\sigma)}\bigg)^{\frac{1}{p-1}} \mathrm{d} \sigma.
		$$
		\item for a.e. $t\in (0,u^\sharp(\bar{s}))$,
		\begin{equation}\label{rigeq2}
			1 =\frac{-\mu^{\prime}(t)}{\mathrm{AVR}_{\mathfrak{m}}^{\frac{q}{N}}\mathcal{I}_{ N}(\mu(t))}\bigg(\frac{F(\mu(t))}{\mathcal{I}_{ N}(\mu(t))}\bigg)^{\frac{1}{p-1}}.
		\end{equation}
	\end{itemize}
	Therefore, for any $s\in (\bar{s},\mathfrak{m}(\Omega))$ we have that
	\begin{align*}
		u^{\sharp}(s)&=\int_0^{u^{\sharp}(s)}\frac{-\mu^{\prime}(t)}{\mathrm{AVR}_{\mathfrak{m}}^{\frac{q}{N}}\mathcal{I}_{ N}(\mu(t))}\bigg(\frac{F(\mu(t))}{\mathcal{I}_{ N}(\mu(t))}\bigg)^{\frac{1}{p-1}}\mathrm{d} t\\
		&=\int_{\mu(u^\sharp(s))}^{\mu(0)} \frac{1}{\mathrm{AVR}_{\mathfrak{m}}^{\frac{q}{N}}\mathcal{I}_{ N}(\sigma)}\bigg(\frac{F(\sigma)}{\mathcal{I}_{ N}(\sigma)}\bigg)^{\frac{1}{p-1}} \mathrm{d} \sigma\\
		&=\int_{s}^{\mathfrak{m}(\Omega)} \frac{1}{\mathrm{AVR}_{\mathfrak{m}}^{\frac{q}{N}}\mathcal{I}_{ N}(\sigma)}\bigg(\frac{F(\sigma)}{\mathcal{I}_{ N}(\sigma)}\bigg)^{\frac{1}{p-1}} \mathrm{d} \sigma=\mathrm{AVR}_{\mathfrak{m}}^{-q/N}v^\sharp(s),
	\end{align*}
	as desired.
	
	To prove (ii), since for a.e. $t\in (0,u^\sharp(\bar{s}))$, $(\ref{rigeq2})$ holds and thus $\mathrm{AVR}_{\mathfrak{m}}^{\frac{1}{N}}\mathcal{I}_{N}(\mu(t))  = \operatorname{Per}(\{|u|>t\})$ by $(\ref{thmeq1.0})$. By Theorem $\ref{isorig}$, this implies that $(\mathrm{X}, \mathrm{d}, \mathfrak{m})$ is a Euclidean metric measure cone.
	
	For part (iii), assume that $\bar{x}=0$ (thus $\mathrm{AVR}_{\mathfrak{m}}^{q/N} u^\star=v$ in $[0,r_a]$), $u \in \operatorname{Lip}(\Omega)$ and $|\nabla u| \neq 0$, $\mathfrak{m}$-a.e. in $\{u\neq 0\}$. Putting together the gradient comparison inequality $(\ref{r-grad})$ and the Euclidean P\'{o}lya-Szeg\"{o} inequality $(\ref{rigeq1})$, we have that
	\begin{align*}
		\mathrm{AVR}_{\mathfrak{m}}^{p/N}\int_0^{r_a}|\nabla u^{\star}|^p  \mathrm{d}\mathfrak{m}_{N} &\leq \int_{\Omega}|\nabla u|^p  \mathrm{d}\mathfrak{m}\leq \mathrm{AVR}_{\mathfrak{m}}^{-\frac{p}{(p-1)N}} \int_0^{r_a}|\nabla v|^p \mathrm{d}\mathfrak{m}_{N}\\
		&=\mathrm{AVR}_{\mathfrak{m}}^{p/N} \int_0^{r_a}|\nabla u^{\star}|^p \mathrm{d}\mathfrak{m}_{N}.
	\end{align*}
	This implies that equality in the Euclidean P\'{o}lya-Szeg\"{o} inequality is achieved. By the rigidity in the Euclidean P\'{o}lya-Szeg\"{o} inequality, $u$ is radial.
\end{proof}

\section{Almost rigidity}
In this section, we will prove a stable version of the rigidity result, namely Theorem $\ref{alrig}$. We begin in Subsection $\ref{alrig1}$ with the necessaries preliminaries concerning the weak/strong convergence for real-valued functions defined in varying $L^p/W^{1,p}$-spaces. In Subsection $\ref{alrig2}$, we prove a result about compact embeddings from $W^{1,p}$ spaces into $L^p$ spaces (Theorem $\ref{compact}$). In Subsection $\ref{alrig3}$, we prove the main almost rigidity result of this note stated in Theorem $\ref{alrig}$. For this result we apply the compactness theorem under pmGH-convergence and fully exploit the rigidity result on the limit space developed in the previous section. Finally, in Subsection $\ref{alrig4}$, we will state a continuity result for the solutions to the Poisson problem (Theorem $\ref{conPoi}$) .

\subsection{Weak and strong $L^p/W^{1,p}$-Convergence w.r.t. varying measures}\label{alrig1}
$\quad$ In this subsection, we assume that $(\mathrm{Y},\mathrm{d})$ is a complete and separable metric space and $p\in (1,\infty)$, and that $\{\mathfrak{m}_n\}_{n\in \overline{\mathbb{N}}_+}\subset \mathcal{M}_{\mathrm{loc}}(\mathrm{Y})$ is such that $\{\mathfrak{m}_n\}_{n}$ weakly converges to $\mathfrak{m}_\infty$, i.e.
$$
\lim_{n \rightarrow \infty} \int \phi \mathrm{d}\mathfrak{m}_n=  \int \phi  \mathrm{d}\mathfrak{m}_\infty,\quad \text { for every } \phi \in \mathrm{C}_{\mathrm{bs}}(\mathrm{Y}).
$$

Recall that whenever we fix $p\in (1,\infty)$, the letter $q$ is automatically defined as the conjugate exponent by the relation
$p^{-1}+q^{-1}=1$ and thus ranging in $(1, \infty)$.

\begin{definition}[$L^p$-Convergence]
	Let $f_n \in L_{\mathrm{loc}}^1(\mathrm{Y},\mathfrak{m}_n)$, $n\in \overline{\mathbb{N}}_+$. We say that $\{f_n\}_n$ weakly converges to $f_\infty$ if 
	\begin{equation}\label{1}
		\lim_{n \rightarrow \infty} \int \phi f_n \mathrm{d}\mathfrak{m}_n=  \int \phi f_\infty \mathrm{d}\mathfrak{m}_\infty,\quad \text { for every } \phi \in \mathrm{C}_{\mathrm{bs}}(\mathrm{Y}).
	\end{equation}
	We say that $\{f_n\}_n$ $L^p$-weakly converges to $f_\infty$ if $(\ref{1})$ holds and moreover 
	\begin{equation}\label{2}
		\sup_{n\in \mathbb{N}_+}\int |f_n|^p \mathrm{d}\mathfrak{m}_n<\infty.
	\end{equation}
	If $\{f_n\}_n$ is $L^p$-weakly converging to $f_\infty$ and furthermore
	\begin{equation}\label{3}
		\limsup_{n\rightarrow \infty}\int |f_n|^p \mathrm{d}\mathfrak{m}_n\leq \int |f_\infty|^p \mathrm{d}\mathfrak{m}_\infty<\infty,
	\end{equation}
	then we say that it is $L^p$-strongly converging to $f_\infty$.	
\end{definition} 

The proof of the following result is very similar to the proof in the case when $\mathfrak{m}_n=\mathfrak{m}_\infty$ holds for every $n\in \mathbb{N}_+$.
\begin{proposition}\label{alrigprop5}
	If $\{f_n\}_n$ $L^p$-weakly converges to $f_\infty$, then 
	$$\|f_{\infty}\|_{L^p(\mathrm{Y},\mathfrak{m}_{\infty})} \leq \liminf_{n\rightarrow \infty}\|f_n\|_{L^p(\mathrm{Y},\mathfrak{m}_n)}.$$
\end{proposition}
Strongly inspired by Honda's work in \cite{H}, we introduce the following notion:
\begin{definition}
	Let $f_\infty\in L^p(\mathrm{Y},\mathfrak{m}_\infty)$ and let $f_{n, m} \in L^p(\mathrm{Y},\mathfrak{m}_n)$ for every $n\in \overline{\mathbb{N}}_+$ and $m\in \mathbb{N}_+$. We say that $\{f_{n, m}\}_{n, m}$ is an $L^p$-approximate sequence of $f_\infty$ if the following two conditions hold$:$\\
	$(1)$ For every $m$, $\{f_{n, m}\}_{n=1}^\infty$ $L^p$-strongly converges to $f_{\infty, m}$.\\
	$(2)$ $\|f_\infty-f_{\infty, m}\|_{L^p(\mathrm{Y},\mathfrak{m}_\infty)} \rightarrow 0$ as $m \rightarrow \infty$. 
\end{definition}

\begin{proposition}\label{alrigprop1}
	Let $f_{\infty} \in L^p(\mathrm{Y},\mathfrak{m}_\infty)$ and let $\{f_{n, m}\}_{n, m}$, $\{\hat{f}_{n, m}\}_{n, m}$ be $L^p$-approximate sequences of $f_{\infty}$. Then
	$$
	\limsup_{m\rightarrow \infty} \limsup_{n \rightarrow \infty}\|f_{n, m}-\hat{f}_{n, m}\|_{L^p(\mathrm{Y},\mathfrak{m}_n)} =0 .
	$$
\end{proposition}
\begin{proof}
	If $p\geq 2$, then Clarkson's inequality for $p\geq 2$ yields
	\begin{align*}
		2^{p-1}(\|f_{n, m}\|_{L^p}^p+\|\hat{f}_{n, m}\|_{L^p}^p) \geq\|f_{n, m}+\hat{f}_{n, m}\|_{L^p}^p+\|f_{n, m}-\hat{f}_{n, m}\|_{L^p}^p.
	\end{align*}
	By letting $n\rightarrow \infty$ and then letting $m\rightarrow \infty$, we obtain the desired result. Similarly, the assertion of the case $1<p<2$ follows from Clarkson's inequality for $1<p<2$.
\end{proof}
\begin{proposition}\label{alrigprop6}
	For every $f_\infty\in L^p(\mathrm{Y},\mathfrak{m}_\infty)$, there exists an $L^p$-approximate sequence $\{f_{n,m}\}_{n,m}$ of $f_\infty$. Moreover, for such sequence, we have that $\{\eta f_{n,m}\}_{n,m}$ is an $L^p$-approximate sequence of $\eta f_\infty$ for any $\eta\in \mathrm{C}_{\mathrm{b}}(\mathrm{Y})$. 
	
\end{proposition}
\begin{proof}
	Since $\mathrm{Lip}_{\mathrm{bs}}(\mathrm{Y})$ is dense in $L^p(\mathrm{Y},\mathfrak{m}_\infty)$, we can find a sequence $\{f_{\infty,m}\}_m\subset \mathrm{Lip}_{\mathrm{bs}}(\mathrm{Y})$ such that $\|f_{\infty,m}-f_\infty\|_{L^p(\mathrm{Y},\mathfrak{m}_\infty)}\rightarrow 0$ as $m \rightarrow \infty$. Then, we define $f_{n,m}=f_{\infty,m}$ for every $n\in \overline{\mathbb{N}}_+$ and $m\in \mathbb{N}_+$. Since $\{\mathfrak{m}_n\}_{n}$ weakly converges to $\mathfrak{m}_\infty$, it is easy to check that $\{\eta f_{n,m}\}_{n,m}$ is an $L^p$-approximate sequence of $\eta f_\infty$ for any $\eta\in \mathrm{C}_{\mathrm{b}}(\mathrm{Y})$ and in particular, $\{f_{n,m}\}_{n,m}$ is an $L^p$-approximate sequence of $f_\infty$.
\end{proof}
We say that $\{f_{n,m}\}$ is a canonical $L^p$-approximate sequence of $f_\infty$ if it is an $L^p$-approximate sequence of $f_\infty$ and moreover for every $n\in \overline{\mathbb{N}}_+$ and $m\in \mathbb{N}_+$, $f_{n,m}=f_{\infty,m}\in \mathrm{Lip}_{\mathrm{bs}}(\mathrm{Y})$.
\begin{proposition}\label{alrigprop2}
	Let $f_n ,g_n\in L^p(\mathrm{Y},\mathfrak{m}_n)$, $n\in \overline{\mathbb{N}}_+$. Assume that 
	\begin{equation}\label{alrigeq1}
		\limsup_{m \rightarrow \infty} \limsup _{n \rightarrow \infty}\|f_n-f_{n, m}\|_{L^p(\mathrm{Y},\mathfrak{m}_n)} =0,
	\end{equation} 
	for some $($and thus any$)$ $L^p$-approximate sequence $\{f_{n,m}\}$ of $f_\infty$ and that $\{g_n\}_n$ $L^q$-weakly converges to $g_\infty$, then
	$$
	\lim _{n \rightarrow \infty} \int f_n g_n \mathrm{d} \mathfrak{m}_n=\int f_\infty g_\infty \mathrm{d} \mathfrak{m}_\infty.
	$$
\end{proposition} 
\begin{proof}
	By Proposition $\ref{alrigprop1}$, we know that $(\ref{alrigeq1})$ holds for any $L^p$-approximate sequence $\{f_{n,m}\}$ of $f_\infty$ if if it holds for some such sequence. Now, let $\{f_{n,m}\}$ be a canonical $L^p$-approximate sequence of $f_\infty$. Since $\{g_n\}_n$ $L^q$-weakly converges to $g_\infty$, we have that for every $m$,
	$$
	\lim_{n \rightarrow \infty} \int f_{\infty,m} g_n \mathrm{d} \mathfrak{m}_n=\int f_{\infty,m} g_\infty \mathrm{d} \mathfrak{m}_\infty.
	$$ 
	On the other hand, the triangle inequality and the H\"{o}lder inequality yield
	\begin{align*}
		\bigg|	\int f_n g_n \mathrm{d} \mathfrak{m}_n-\int f_\infty g_\infty \mathrm{d} \mathfrak{m}_\infty\bigg|&\leq \|f_n-f_{n, m}\|_{L^p(\mathrm{Y},\mathfrak{m}_n)}\|g_n\|_{L^q(\mathrm{Y},\mathfrak{m}_n)}\\
		&\quad +\bigg|	\int f_{\infty,m} g_n  \mathrm{d} \mathfrak{m}_n-\int f_{\infty,m} g_\infty \mathrm{d} \mathfrak{m}_\infty\bigg|\\
		&\quad+  \|f_{\infty,m}-f_{\infty}\|_{L^p(\mathrm{Y},\mathfrak{m}_\infty)}\|g_\infty\|_{L^q(\mathrm{Y},\mathfrak{m}_\infty)}.
	\end{align*}
	By letting $n\rightarrow \infty$ and then letting $m\rightarrow \infty$, we obtain the desired result.
\end{proof}
\begin{proposition}\label{alrigprop3}
	Let $f_n \in L^p(\mathrm{Y},\mathfrak{m}_n)$, $n\in \overline{\mathbb{N}}_+$. Then the following are equivalent.\\
	$(1)$ $\{f_n\}_n$ $L^p$-strongly converges to $f_\infty$.\\
	$(2)$ For some $($and thus any$)$ $L^p$-approximate sequence $f_{n,m}$ of $f_\infty$, it holds that $$
	\limsup_{m \rightarrow \infty} \limsup _{n \rightarrow \infty}\|f_n-f_{n, m}\|_{L^p(\mathrm{Y},\mathfrak{m}_n)} =0.
	$$
	$(3)$ For every $\zeta \in \mathrm{C}(\mathrm{Y} \times \mathbb{R})$ with $|\zeta(y, r)| \leq \varphi(y)+C|r|^p$, $\varphi\in \mathrm{C}_{\mathrm{bs}}(\mathrm{Y})$, $C \geq 0$, we have that
	\begin{equation}\label{4}
		\int \zeta(y, f_n(y)) \mathrm{d} \mathfrak{m}_n(y) \rightarrow \int \zeta(y, f_\infty(y)) \mathrm{d} \mathfrak{m}_\infty(y).
	\end{equation}		
\end{proposition}
\begin{proof}
	Since the proof of (1) $\Rightarrow$ (2) is very similar to the proof of Proposition $\ref{alrigprop1}$, we omit it. 
	
	(2) $\Rightarrow$ (1): Let $\{f_{n, m}\}_{n, m}$ be an $L^p$-approximate sequence of $f_\infty$. Then we have that
	\begin{align*}
		\limsup_{n \rightarrow \infty}\|f_n\|_{L^p(\mathrm{Y},\mathfrak{m}_n)}&\leq \limsup_{m \rightarrow \infty}\limsup _{n \rightarrow \infty}\|f_n-f_{n, m}\|_{L^p(\mathrm{Y},\mathfrak{m}_n)}+\limsup_{m \rightarrow \infty}\limsup_{n \rightarrow \infty}\|f_{n, m}\|_{L^p(\mathrm{Y},\mathfrak{m}_n)}\\
		&\leq \limsup_{m \rightarrow \infty}\|f_{\infty, m}\|_{L^p(\mathrm{Y},\mathfrak{m}_\infty)}=\|f_{\infty}\|_{L^p(\mathrm{Y},\mathfrak{m}_\infty)}.
	\end{align*}
	On the other hand, by Proposition $\ref{alrigprop2}$, we have that
	\begin{align*}
		\lim_{n \rightarrow \infty}\int f_n\phi \mathrm{d}\mathfrak{m}_n=\int f_\infty\phi \mathrm{d}\mathfrak{m}_\infty,
	\end{align*}
	for every $\phi \in \mathrm{C}_{\mathrm{bs}}(\mathrm{Y})$.
	
	(3) $\Rightarrow$ (1): Fix $\phi \in \mathrm{C}_{\mathrm{bs}}(\mathrm{Y})$. Then (1) follows by taking $\zeta(y,r)=\phi(y) r$ and $\zeta(y,r)=|r|^p$ in $(\ref{4})$, respectively.
	
	In order to prove (1) $\Rightarrow$ (3), we need the following Lemma $\ref{alriglem1}$.
\end{proof}

\begin{lemma}\label{alriglem1}
	Let $f_n\in L^p(\mathrm{Y},\mathfrak{m}_n)$, $n\in \overline{\mathbb{N}}_+$. If $\{f_n\}_n$ $L^p$-strongly converges to $f_\infty$, then for every $\eta\in \mathrm{C}_{\mathrm{b}}(\mathrm{Y})$, $\{\eta f_n\}_n$ $L^p$-strongly converges to $\eta f_\infty$. In particular, for any lower semi-continuous function $g:\mathrm{Y}\rightarrow [0,\infty)$ and $\phi\in \mathrm{C}_{\mathrm{b}}(\mathrm{Y})$ with $\phi\geq 0$, it holds that
	$$
	\int \phi|f_\infty|^p\mathrm{d}\mathfrak{m}_\infty= \lim_{n\rightarrow \infty} \int \phi|f_n|^p\mathrm{d}\mathfrak{m}_n\text{ and }\int g|f_\infty|^p\mathrm{d}\mathfrak{m}_\infty\leq \liminf_{n\rightarrow \infty} \int g|f_n|^p\mathrm{d}\mathfrak{m}_n.
	$$
\end{lemma} 
\begin{proof}
	The first statement follows from Proposition $\ref{alrigprop6}$ and the equivalence of (1) and (2). And the second statement follows from the first statement and the fact that every non-negative lower semi-continuous function can be represented as the supremum of bounded continuous functions with bounded support.
\end{proof} 

Now, we proceed with the proof of (1) $\Rightarrow$ (3) in Proposition $\ref{alrigprop3}$:
\begin{proof}
	By \cite[Theorem 5.4.4]{AGS08}, we know that the conclusion holds whenever $\mathfrak{m}_n\in \mathscr{P}(\mathrm{Y})$ for every $n\in \overline{\mathbb{N}}_+$. Observe that for any $\mathfrak{n}\in \mathcal{M}_{\mathrm{loc}}(\mathrm{Y})$ and any Borel function $f:\mathrm{Y}\rightarrow \mathbb{R}$, we have that $(\mathrm{Id}\times f)_\sharp \mathfrak{n}\in \mathcal{M}_{\mathrm{loc}}(\mathrm{Y}\times \mathbb{R})$. 
	
	We fix a non-negative function $\eta\in \mathrm{Lip}_{\mathrm{bs}}(\mathrm{Y})$ and set $\tilde{\mathfrak{m}}_n=\eta {\mathfrak{m}}_n$ for every $n\in \overline{\mathbb{N}}_+$. Clearly, $\{\tilde{\mathfrak{m}}_n\}_n$ weakly converges to $\tilde{\mathfrak{m}}_\infty$ and for every $\phi \in \mathrm{C}_{\mathrm{bs}}(\mathrm{Y})$, we have that
	$$
	\int f_n \phi \mathrm{d}\tilde{\mathfrak{m}}_n=\int f_n \phi \eta\mathrm{d}\mathfrak{m}_n\rightarrow \int f_\infty \phi \eta\mathrm{d}\mathfrak{m}_\infty=\int f_\infty \phi \mathrm{d}\tilde{\mathfrak{m}}_\infty.
	$$ 
	By Lemma $\ref{alriglem1}$, we know that $\{f_n\eta^{1/p}\}_n$ $L^p$-strongly converges to $f_\infty\eta^{1/p}$. Hence,  	
	$$
	\limsup_{n\rightarrow \infty}\int|f_n|^p\mathrm{d} \tilde{\mathfrak{m}}_n=\limsup_{n\rightarrow \infty}\int|f_n\eta^{1/p}|^p \mathrm{d}{\mathfrak{m}}_n\leq \int|f_\infty\eta^{1/p}|^p {\mathfrak{m}}_\infty=\int|f_\infty|^p\mathrm{d} \tilde{\mathfrak{m}}_\infty.
	$$ 	
	Therefore, we obtain that $\{f_n\}_n$ $L^p$-strongly converges to $f_\infty$ with respect to the measures $\{\tilde{\mathfrak{m}}_n\}_{n\in \overline{\mathbb{N}}_+}$. By \cite[Theorem 5.4.4]{AGS08}, we have that
	\begin{equation}\label{alrigeq2}
		\int \zeta(y, f_n(y))\eta(y) \mathrm{d} {\mathfrak{m}}_n(y) \rightarrow \int \zeta(y, f_\infty(y))\eta(y)  \mathrm{d} {\mathfrak{m}}_\infty(y),
	\end{equation}	
	for every $\zeta \in \mathrm{C}(\mathrm{Y} \times \mathbb{R})$ with $|\zeta(y, r)| \leq \varphi(y)+C|r|^p$, $\varphi\in \mathrm{C}_{\mathrm{bs}}(\mathrm{Y})$, $C \geq 0$
	
	Finally, we fix a function $\zeta \in \mathrm{C}(\mathrm{Y} \times \mathbb{R})$ with $|\zeta(y, r)| \leq \varphi(y)+C|r|^p$, $\varphi\in \mathrm{C}_{\mathrm{bs}}(\mathrm{Y})$, $C \geq 0$. Without loss of generality, we can assume that $\zeta$ is non-negative. To prove $(\ref{4})$, since $(\ref{alrigeq2})$ holds for every non-negative function $\eta\in \mathrm{Lip}_{\mathrm{bs}}(\mathrm{Y})$, it suffices to show that for every $R>0$,
	$$
	\lim_{R\rightarrow\infty}    \limsup_{n\rightarrow\infty}\int_{B_R(\bar{y}_\infty)^c} \zeta(y, f_n(y))\mathrm{d} {\mathfrak{m}}_n(y)=0.
	$$	
	Indeed, let $\bar{y}\in \mathrm{Y}$ and let $R>0$ be sufficiently large such that $\mathrm{supp}(\varphi)\subset B_R(\bar{y})$. Then by Lemma $\ref{alriglem1}$ again,
	\begin{align*}
		\limsup_{n\rightarrow\infty}\int_{B_R(\bar{y})^c} \zeta(y, f_n(y))\mathrm{d} {\mathfrak{m}}_n(y)&\leq C\limsup_{n\rightarrow\infty}\int_{B_R(\bar{y})^c} |f_n(y)|^p\mathrm{d} {\mathfrak{m}}_n(y)\\
		&\leq C\int_{B_R(\bar{y})^c} |f_\infty|^p\mathrm{d} {\mathfrak{m}}_\infty\rightarrow 0,\text{ as }R\rightarrow \infty.
	\end{align*}  	
\end{proof}
In the following proposition, we collect a list of useful properties of $L^p$-convergence, see also \cite[Proposition 2.18]{NV22}.
\begin{proposition}[Properties of $L^p$-convergence]\label{2.18}
	$(\mathrm{i})$ If $\{f_n\}$ $L^p$-strongly converges to $f_\infty$, then $\{\varphi\circ f_n\}$ $L^p$-strongly converges to $\varphi\circ f_\infty$ for every $\varphi \in \operatorname{Lip}(\mathbb{R})$ with $\varphi(0)=0$.\\
	$(\mathrm{ii})$ If $\{f_n\}$ $($resp. $\{g_n\})$ $L^p$-strongly converges to $f_\infty$ $($resp. $g_\infty)$ and $\alpha\in \mathbb{R}$, then $\{\alpha f_n+g_n\}_n$ $L^p$-strongly converges to $f_\infty+g_\infty$.\\
	$(\mathrm{iii})$ If $f_{\infty} \in L^p(\mathrm{Y},\mathfrak{m}_{\infty})$, then there exists a sequence $n\mapsto f_{n}\in L^p(\mathrm{Y},\mathfrak{m}_{n})$ $L^p$-strongly converging to $f_{\infty}$.\\
	$(\mathrm{iv})$ If $\sup_n\|f_n\|_{L^p(\mathrm{Y},\mathfrak{m}_n)}<\infty$, then, up to a subsequence, $\{f_n\}$ $L^p$-weakly converges to some $f_{\infty} \in L^p(\mathrm{Y},\mathfrak{m}_{\infty})$.\\
	$(\mathrm{v})$ If $\{f_n\}$ $L^p$-strongly converges to $f_{\infty}$ and for each $n$, $f_n\geq 0$, then for every $r \in(1, \infty)$, $\{f_n^{p / r}\}$ $L^r$-strongly converges to $f_{\infty}^{p / r}$.\\
	$(\mathrm{vi})$ Let $r,s \in(1, \infty]$ be such that $r<s$. If $\{f_n\}$ is uniformly bounded in $L^s$ and $L^r$-strongly converges to $f_{\infty}$, then it $L^t$-strongly converges also to $f_{\infty}$ for every $t \in[r, s)$.	
\end{proposition}

Next, we review the concepts of weak and strong convergence in the Sobolev space $W^{1,p}$ with respect to variable reference measures.
\begin{definition}[$W^{1,p}$-Convergence]
	Let $f_n \in W^{1,p}(\mathrm{Y},\mathfrak{m}_n)$, $n\in \overline{\mathbb{N}}_+$. We say that $\{f_n\}_n$ $W^{1,p}$-weakly converges to $f_\infty$ if $\{f_n\}_n$ $L^p$-weakly converges to $f_\infty$ and moreover 
	\begin{equation}\label{5}
		\sup_{n\in \mathbb{N}_+} \int |D f_n|^p\mathrm{d}\mathfrak{m}_n<\infty.
	\end{equation}
	If $\{f_n\}_n$ is $L^p$-strongly converging to $f_\infty$ and furthermore
	\begin{equation}\label{6}
		\limsup_{n\rightarrow \infty}\int |D f_n|^p \mathrm{d}\mathfrak{m}_n\leq \int |D f_\infty|^p \mathrm{d}\mathfrak{m}_\infty<\infty,
	\end{equation}
	then we say that it is $W^{1,p}$-strongly converging to $f_\infty$.			
\end{definition}

\subsection{Compactness theorem}\label{alrig2}
In this subsection, we will deal with PI spaces that are uniformly locally doubling and support a weak local $1$-$p$ Poincar\'{e} inequality. 
\begin{definition}
	The metric measure space $(\mathrm{X},\mathrm{d},\mathfrak{m})$ is said to be uniformly locally doubling if 
	\begin{equation}\label{uniformly locally doubling}
		0<\mathfrak{m}(B_{2r}(x))\leq \mathrm{D}(R)\mathfrak{m}(B_{r}(x)),\quad \text{ for every } x\in \mathrm{X}, \ r\in (0,R],
	\end{equation} 
	for some non-decreasing function $\mathrm{D}:(0,\infty)\rightarrow (0,\infty)$.
	
	We say that $(\mathrm{X},\mathrm{d},\mathfrak{m})$ is a Poincar\'{e} space if $\mathfrak{m}$ is uniformly locally doubling and if, in addition, $(\mathrm{X},\mathrm{d},\mathfrak{m})$ supports a weak local $1$-$1$ Poincar\'{e} inequality, i.e. there exists $\lambda>0$ such that for any Lipschitz function $u:\mathrm{X}\rightarrow \mathbb{R}$, it holds that
	$$
	\fint_{B_r(x)}|u-u_{B_r(x)}|\mathrm{d}\mathfrak{m}\leq \mathrm{P}(R)r\fint_{B_{\lambda r}(x)} \mathrm{lip}(u)\mathrm{d}\mathfrak{m}, \text{ for every } x\in \mathrm{X}, \ r\in (0,R],
	$$
	for some non-decreasing function $\mathrm{P}:(0,\infty)\rightarrow (0,\infty)$, where $u_B:=\fint_Bu \mathrm{d}\mathfrak{m}$.
\end{definition}
\begin{remark}\cite{Ra12,St06}
	If $(\mathrm{X},\mathrm{d},\mathfrak{m})$ is a $\mathrm{CD}(K,N)$ space for some $K\in \mathbb{R},N\in(1,\infty)$ with $\mathrm{supp}[\mathfrak{m}]=\mathrm{X}$, then it is uniformly locally doubling with the function $\mathrm{D}_{K,N}(r)=2^N[r\cosh\sqrt{K^-/(N-1)}]^{N-1}$ and it supports a weak local $1$-$p$ Poincar\'{e} inequality with the function $\mathrm{P}_{K,N}(r)=2^{N+3}[r\cosh\sqrt{K^-/(N-1)}]^{N-1}\exp(2r\sqrt{(N-1)K^-})$.
\end{remark}
We say that a p.m.m.s. $(\mathrm{X},\mathrm{d},\mathfrak{m},\bar{x})$ is normalized if $\int_{B_1(\bar{x})} 1-\mathrm{d}(\cdot, \bar{x}) \mathrm{d}\mathfrak{m}=1$. Obviously, given any p.m.m.s. $(\mathrm{X},\mathrm{d},\mathfrak{m},\bar{x})$ there exists a unique $c>0$ such that $(\mathrm{X},\mathrm{d},c\mathfrak{m},\bar{x})$ is normalized, namely $c:=(\int_{B_1(\bar{x})} 1-\mathrm{d}(\cdot, \bar{x}) \mathrm{d}\mathfrak{m})^{-1}$.

We denote by $\mathscr{M}_{\mathrm{D}(\cdot)}$ the class of (isomorphism classes of) normalized p.m.m.s. fulfilling $(\ref{uniformly locally doubling})$ for a given non-decreasing $\mathrm{D}:(0, \infty) \rightarrow(0, \infty)$. Recall that two p.m.m.s. $(\mathrm{X},\mathrm{d},\mathfrak{m},\bar{x})$, $(\mathrm{X}^{\prime},\mathrm{d}^{\prime},\mathfrak{m}^{\prime},\bar{x}^{\prime})\in \mathscr{M}_{\mathrm{D}(\cdot)}$ are said to be isomorphic if there exists an isometry $T:(\mathrm{X}, \mathrm{d}) \rightarrow(\mathrm{X}^{\prime}, \mathrm{d}^{\prime})$ such that $T_{\sharp} \mathfrak{m}=\mathfrak{m}^{\prime}$ and $T(\bar{x})=\bar{x}^{\prime}$.

\begin{definition}[Pointed measured Gromov-Hausdorff convergence]  
	Let $\mathrm{D}:(0, \infty) \rightarrow(0, \infty)$ be a non-decreasing function and let $(\mathrm{X}_n,\mathrm{d}_n,\mathfrak{m}_n,\bar{x}_n)\in \mathscr{M}_{\mathrm{D}(\cdot)}$, $n\in \overline{\mathbb{N}}_+$. We say that $\{(\mathrm{X}_n,\mathrm{d}_n,\mathfrak{m}_n,\bar{x}_n)\}_{n\in {\mathbb{N}}_+}$ converges to $(\mathrm{X}_\infty,\mathrm{d}_\infty,\mathfrak{m}_\infty,\bar{x}_\infty)$ in the pointed measured Gromov-Hausdorff $(\mathrm{pmGH}$ for short$)$ sense if there exists a complete and separable metric space $(\mathrm{Z}, \mathrm{d}_{\mathrm{Z}})$ and a sequence of isometric embeddings $\{\iota_n:(\mathrm{X}_n,\mathrm{d}_n)\rightarrow (\mathrm{Z}, \mathrm{d}_{\mathrm{Z}})\}_{n\in \overline{\mathbb{N}}_+}$ such that for every $\epsilon>0$ and $R>0$ there exists $N\in {\mathbb{N}}_+$ such that for all $n\geq N$,
	$$
	\iota_\infty(\bar{B}_R^{\mathrm{d}_\infty}(\bar{x}_\infty))\subset B^{\mathrm{Z}}_\epsilon(\iota_n(\bar{B}_R^{\mathrm{d}_n}(\bar{x}_n)))\quad\text{ and }\quad \iota_n(\bar{B}_R^{\mathrm{d}_n}(\bar{x}_n))\subset B^{\mathrm{Z}}_\epsilon(\iota_\infty(\bar{B}_R^{\mathrm{d}_\infty}(\bar{x}_n)))
	$$
	where $B_{\epsilon}^{\mathrm{Z}}(A):=\{z \in \mathrm{Z}: \mathrm{d}_{\mathrm{Z}}(z, A)<\epsilon\}$ for every subset $A \subset \mathrm{Z}$, and $$
	\lim_{n\rightarrow \infty}\int \varphi \mathrm{d}(\iota_n)_{\sharp}\mathfrak{m}_n= \int \varphi \mathrm{d}(\iota_{\infty})_{\sharp}\mathfrak{m}_{\infty},\quad \text{ for every }  \varphi \in {\mathrm{C}}_{\mathrm{bs}}(\mathrm{Z}).
	$$
\end{definition}
Sometimes in the following, for simplicity of notation, we will identify the spaces $\mathrm{X}_n$, $n\in \overline{\mathbb{N}}_+$, with their isomorphic copies $\iota_n(\mathrm{X}_n) \subset \mathrm{Z}$.

It is obvious that this is in fact a notion of convergence for isomorphism classes of p.m.m.s., moreover it is induced by a metric (see e.g. \cite{GMS} for details):
\begin{proposition}\label{alrigprop7}
	Let $\mathrm{D}:(0, \infty) \rightarrow(0, \infty)$ be a non-decreasing function. Then there exists a distance $\mathrm{d}_{\mathrm{pmGH}}$ on $\mathscr{M}_{\mathrm{D}(\cdot)}$ for which converging sequences are precisely those converging in the $\mathrm{pmGH}$ sense. Furthermore, the metric space $(\mathscr{M}_{\mathrm{D}(\cdot)}, \mathrm{d}_{\mathrm{pmGH}})$ is compact.	
\end{proposition}

Notice that the compactness of $(\mathscr{M}_{\mathrm{D}(\cdot)}, \mathrm{d}_{\mathrm{pmGH}})$ follows by the standard argument of Gromov: the measures of spaces in $\mathscr{M}_{\mathrm{D}(\cdot)}$ are uniformly locally doubling, hence balls of given radius around the reference points are uniformly compact in the GH-topology. Then weak compactness of the measures follows using the doubling condition again and the fact that they are normalized.

\vspace{0.5cm}

Next, we recall the weak local $p^*$-$p$ Poincar\'{e} inequalities, which plays a central role in the proof of the compactness Theorem $\ref{compact}$. By \cite[Theorem 5.1]{HK00}, we have that $(\ref{alrigeq3})$ holds for Lipschitz functions. By the energy density, it extends to $W^{1,p}$ functions. 
\begin{proposition}\label{alrigprop4}
	Let $(\mathrm{X}, \mathrm{d}, \mathfrak{m})$ be a Poincar\'{e} space and let $p\in (1,\infty)$. Assume that $\mathfrak{m}$ satisfies
	$$
	\mathrm{C}(R)\frac{\mathfrak{m}(B_{r_1}(x))}{\mathfrak{m}(B_{r_2}(x))}\geq \frac{r_1^N}{r_2^N},\quad \text{ for every } x\in \mathrm{X},\ 0<r_1<r_2\leq R,
	$$
	for some non-decreasing function $\mathrm{C}:(0,\infty)\rightarrow (0,\infty)$ and some constant $N\in (1,\infty)$. 
	
	Then for every $R>0$, there exists a constant $C=C(p,N,\mathrm{D}(4\lambda R),\mathrm{P}(2R),\mathrm{C}((2\lambda+1)R))$ such that for every $u\in W^{1,p}(\mathrm{X},\mathfrak{m})$, $x\in \mathrm{X}$ and $r\in (0,R]$,
	\begin{equation}\label{alrigeq3}
		\bigg(\fint_{B_r(x)}|u-u_{B_r(x)}|^{p^*}\mathrm{d}\mathfrak{m}\bigg)^{1/p^*}\leq Cr\bigg(\fint_{B_{(2\lambda+1) r}(x)} |Du|^p\mathrm{d}\mathfrak{m}\bigg)^{1/p},
	\end{equation}
	where $p^*=Np/(N-p)$ if $N>p$, $p^*$ can be any power in $[1,\infty)$ if $N\in (1,p]$.
\end{proposition}
\begin{remark}
	If $(\mathrm{X},\mathrm{d},\mathfrak{m})$ is a $\mathrm{CD}(K,N)$ space for some $K\in \mathbb{R},N\in(1,\infty)$ with $\mathrm{supp}[\mathfrak{m}]=\mathrm{X}$, then 
	$$
	\mathrm{C}_{K,N}(R)\frac{\mathfrak{m}(B_{r_1}(x))}{\mathfrak{m}(B_{r_2}(x))}\geq \frac{r_1^N}{r_2^N},\quad \text{ for every } x\in \mathrm{X},\ 0<r_1<r_2\leq R,
	$$
	where $\mathrm{C}_{K,N}(r)=\exp[r\sqrt{K^-(N-1)}]$. This is a corollary of the Bishop-Gromov inequality\cite{St06}.
\end{remark}

Finally, we prove the main compactness result. To this aim, we make the following assumptions:
\begin{assumption}\label{alrigassu1}
	Let $p\in (1,\infty)$.
	\begin{itemize}
		\item $\mathrm{D},\mathrm{P},\mathrm{C}:(0,\infty)\rightarrow (0,\infty)$ are three non-decreasing function. 
		\item The sequence $\{\mathscr{X}_n:=[\mathrm{X}_n, \mathrm{d}_n, \mathfrak{m}_n,\bar{x}_n]\}_{n\in \bar{\mathbb{N}}_+}\subset \mathscr{M}_{\mathrm{D}(\cdot)}$ is such that $$\mathrm{d}_{\mathrm{pmGH}}(\mathscr{X}_n,\mathscr{X}_\infty)\rightarrow 0.$$  
		\item For all $z\in \mathrm{X}_\infty$ and $r>0$, $\mathfrak{m}_\infty(\bar{B}_r(z)\backslash {B}_r(z))=0$.
		\item For all $n\in \mathbb{N}_+$, $\mathscr{X}_n$ supports a weak local $1$-$1$ Poincar\'{e} inequality with $\mathrm{P}$ and $\mathfrak{m}_n$ satisfies
		$$
		\mathrm{C}(R)\frac{\mathfrak{m}_n(B_{r_1}(x))}{\mathfrak{m}_n(B_{r_2}(x))}\geq \frac{r_1^N}{r_2^N},\quad \text{ for every } x\in \mathrm{X}_n,\ 0<r_1<r_2\leq R.
		$$
		\item There is a proper metric space $(\mathrm{Y},\mathrm{d})$ containing isometrically all the $(\mathrm{X}_n, \mathrm{d}_n)$, $n\in \bar{\mathbb{N}}_+$.
		\item The sequence $\{\mathfrak{m}_n\}_{n}$ weakly converges to $\mathfrak{m}_\infty$ and $\mathrm{d}(\bar{x}_n,\bar{x}_\infty)\rightarrow 0$.
		\item Let $x_n\in \mathrm{X}_n$, $n\in \overline{\mathbb{N}}_+$, be such that $\mathrm{d}({x}_n,{x}_\infty)\in(0,1)$ for every $n$ and $\mathrm{d}({x}_n,{x}_\infty)\rightarrow 0$.
		\item Let $\{R_n\}_{n\in \overline{\mathbb{N}}_+}\subset (0,\infty)$ be such that $|R_n-R_\infty|\in(0,1)$ for every $n$ and $R_n\rightarrow R_\infty$.
	\end{itemize}
\end{assumption}

\begin{lemma}\label{alriglem2}
	Let $f_n\in W^{1, p}(\mathrm{Y},\mathfrak{m}_n)$, $n\in \mathbb{N}_+$, be non-negative functions with 
	$$\sup_n\|f_n\|_{W^{1, p}(\mathrm{Y},\mathfrak{m}_n)}<\infty.$$ Then all $L^p$-weak limits of subsequences of $\{f_n\}_n$ belongs to $W^{1, p}(\mathrm{Y}, \mathfrak{m}_\infty)$.
\end{lemma}
\begin{proof}
	Let $f_\infty \in L^p(\mathrm{Y}, \mathfrak{m}_\infty)$ be a $L^p$-weak limit of some subsequence of $\{f_n\}_n$ and we need to prove that $f_\infty\in W^{1, p}(\mathrm{Y}, \mathfrak{m}_\infty)$. Without loss of generality, we assume that $\{f_n\}_n$ $L^p$-weakly converges to $f_\infty$. By Proposition $\ref{alrigprop5}$ and $\sup_n\||Df_n|\|_{ L^p(\mathrm{Y},\mathfrak{m}_n)}<\infty$, up to a not relabeled subsequence, we can assume that $\{|Df_n|\}_n$ $L^p$-weakly converges to some function $g \in L^p(\mathrm{Y}, \mathfrak{m}_\infty)$. For simplify, we denote by $v_{z, r}$ the mean value of the function $v$ on $B_r(z)$.
	
	Fix $\tau>0$ and let $s_k=2^{-k} r$ for every $r\in (0,\tau]$ and $k\in \mathbb{N}$. Let $z_1,z_2\in \mathrm{X}_\infty$ be Lebesgue points for $f_\infty$ with $\mathrm{d}(z_1,z_2)= r$ and thus we have that $(f_\infty)_{z_1, s_k} \rightarrow f_\infty(z_1)$ and $(f_\infty)_{z_2, s_k} \rightarrow f_\infty(z_2)$ as $k \rightarrow \infty$. Since $\mathscr{X}_n\rightarrow\mathscr{X}_\infty$, we can find two sequences $n\mapsto z_{1,n}\in \mathrm{X}_n$ and $n\mapsto z_{2,n}\in \mathrm{X}_n$ such that $\mathrm{d}(z_{1,n},z_{2,n})<2r$ for all $n$ and $\mathrm{d}(z_{i,n},z_i)\rightarrow 0$  for $i=1,2$. 
	
	Fix $n$. Since $\mathfrak{m}_n$ is uniformly locally doubling and the space $\mathscr{X}_n$ supports $1$-$1$ Poincar\'{e} inequality, we have that
	\begin{align}
		|(f_n)_{z_{1,n}, s_k}-(f_n)_{z_{1,n}, 2 s_k}| & \leq \fint_{B_{s_k}(z_{1,n})}|f_n-(f_n)_{z_{1,n}, 2 s_k}| \mathrm{d} \mathfrak{m}_n\notag\\
		&\leq \mathrm{D}(\tau) \fint_{B_{2s_k}(z_{1,n})}|f_n-(f_n)_{z_1, 2 s_k}| \mathrm{d} \mathfrak{m}_n \notag\\
		& \leq \mathrm{D}(\tau) \mathrm{P}(2\tau) 2s_k\fint_{B_{2\lambda s_k}(z_{1,n})} |Df_n| \mathrm{d} \mathfrak{m}_n,\label{alrigeq6}
	\end{align}
	and
	\begin{align}
		|(f_n)_{z_{1,n}, r}-(f_n)_{z_{2,n}, 2 r}| & \leq \fint_{B_{r}(z_{1,n})}|(f_n)-(f_n)_{z_{2,n}, 2 r}| \mathrm{d} \mathfrak{m}_n\notag\\
		& \leq \mathrm{D}(2\tau)^2\fint_{B_{2r}(z_{2,n})}|(f_n)-(f_n)_{z_{2,n}, 2 r}| \mathrm{d} \mathfrak{m}_n \notag\\
		& \leq \mathrm{D}(2\tau)^2 \mathrm{P}(2\tau) 2r\fint_{B_{2\lambda r}(z_{2,n})} |Df_n| \mathrm{d} \mathfrak{m}_n.\label{alrigeq7}
	\end{align}
	
	Now, observe that it holds that
	$$
	(f_n)_{z_{i,n}, r}\xrightarrow{n} (f_\infty)_{z_{i}, r},\quad i=1,2,
	$$
	by the facts that $\{f_n\}_n$ weakly converges to $f_\infty$, the sequence $\{\mathfrak{m}_n\}_{n}$ weakly converges to $\mathfrak{m}_\infty$ and $\mathfrak{m}_\infty(\bar{B}_s(z)\backslash {B}_s(z))=0$ for all $z\in \mathrm{X}_\infty$ and $s>0$. The similar assertions hold for $\{|Df_n|\}$ and $g$. Thereofore, by letting $n\rightarrow \infty$ in $(\ref{alrigeq6})$ and $(\ref{alrigeq7})$, we obtain that
	\begin{align*}
		|(f_\infty)_{z_{1}, s_k}-(f_\infty)_{z_{1}, 2 s_k}| \leq \mathrm{D}(\tau) \mathrm{P}(2\tau) 2s_k\fint_{B_{2\lambda s_k}(z_{1})} g \mathrm{d} \mathfrak{m}_\infty
		\leq  \mathrm{D}(\tau) \mathrm{P}(2\tau) 2s_kM_{2\lambda\tau}(g)(z_1),
	\end{align*}
	and
	\begin{align*}
		|(f_\infty)_{z_{1}, r}-(f_\infty)_{z_{2}, 2 r}| \leq \mathrm{D}(2\tau)^2 \mathrm{P}(2\tau) 2r\fint_{B_{2\lambda r}(z_{2})} g \mathrm{d} \mathfrak{m}_\infty
		\leq \mathrm{D}(2\tau)^2 \mathrm{P}(2\tau) 2rM_{2\lambda\tau}(g)(z_2),
	\end{align*}
	where
	$$
	M_{2\lambda\tau}(g)(z_1):=\sup_{0<s\leq 2\lambda\tau}\fint_{B_{s}(z_{1})} g \mathrm{d} \mathfrak{m}_\infty.
	$$
	Hence,
	$$
	|(f_\infty)_{z_{1}, r}-f_\infty(z_1)| \leq \sum_{k=1}^{\infty}|(f_\infty)_{z_{1}, s_k}-(f_\infty)_{z_{1}, 2 s_k}| \leq \mathrm{D}(\tau) \mathrm{P}(2\tau) 2r M_{2\lambda\tau}(g)(z_1).
	$$
	Similarly, we can deduce that 
	$$
	|(f_\infty)_{z_{2}, 2r}-f_\infty(z_2)| \leq \mathrm{D}(\tau) \mathrm{P}(2\tau)4r M_{4\lambda\tau}(g)(z_2).
	$$
	Since $\mathrm{d}(z_1,z_2)=r$, we have that
	\begin{align*}
		|f_\infty(z_1)-f_\infty(z_2)| & \leq|f_\infty(z_1)-(f_\infty)_{z_{1}, r}|+|(f_\infty)_{z_{1}, r}-(f_\infty)_{z_{2}, 2r}|+|(f_\infty)_{z_2, 2 r}-f_\infty(z_2)| \\
		& \leq 6\mathrm{D}(2\tau)^2 \mathrm{P}(2\tau)\mathrm{d}(z_1,z_2)\big[  M_{4\lambda\tau}(g)(z_1)+  M_{4\lambda\tau}(g)(z_2)\big].
	\end{align*}
	This implies that $f_\infty$ is in the Hajlasz-Sobolev space, and therefore also in $W^{1,p}(\mathrm{Y},\mathfrak{m}_\infty)$ by \cite[Proposition 2.15]{GT21}.
\end{proof}

\begin{theorem}[Compactness Theorem]\label{compact}
	Under the assumption $\ref{alrigassu1}$, let $n\mapsto f_n\in W_0^{1, p}(B_{R_n}(x_n),\mathfrak{m}_n)$ be such that $\sup_n\|f_n\|_{W^{1, p}(\mathrm{Y}, \mathfrak{m}_n)}<\infty$. Then the sequence $\{f_n\}_n$ has a subsequence $\{f_{n_k}\}_k$ that $L^p$-strongly converges to some function $f_\infty\in W^{1, p}(\mathrm{Y},\mathfrak{m}_\infty)$.
\end{theorem}

\begin{proof}
	Without loss of generality, we can assume that $f_n\geq 0$ for every $n\in {\mathbb{N}}_+$. By Proposition $\ref{alrigprop5}$ and Lemma $\ref{alriglem2}$, up to a not relabeled subsequence, we can assume that $\{f_n\}_n$ $L^p$-weakly converges to some function $f_\infty\in  W^{1, p}(\mathrm{Y},\mathfrak{m}_\infty)$. Since $\{f_{n}\}_n$ weakly converges to $f_\infty$, it is not difficult to check that $f\in \hat{W}_0^{1, p}(B_{R_\infty}(x_\infty), \mathfrak{m}_\infty)$. 
	
	By Proposition $\ref{alrigprop6}$, we can find an $L^p$-approximate sequence $\{f_{n,m}\}$ of $f_\infty$ such that $f_{n,m}=f_{\infty,m}\in \mathrm{Lip}_{\mathrm{c}}(\mathrm{Y})$ for every $n\in \overline{\mathbb{N}}_+$ and $m\in {\mathbb{N}}_+$. Since $f_\infty\in W^{1, p}(\mathrm{Y}, \mathfrak{m}_\infty)$ and $f_\infty=0$, $\mathfrak{m}_\infty$-a.e. in $\mathrm{Y}\backslash \bar{B}_{R_\infty}(x_\infty)$, we can also require that $\mathrm{supp}(f_{\infty,m})\subset B_{{R_\infty}+2}(x_\infty)$ for every $m$ and
	$$ \int\mathrm{lip}_a^p(f_{\infty,m})\mathrm{d}\mathfrak{m}_\infty\xrightarrow{m\rightarrow \infty} \int |D f_\infty|^p\mathrm{d}\mathfrak{m}_\infty, 
	$$
	where $\mathrm{lip}_a(g)$ is the asymptotic Lipschitz constant of a function $g:\mathrm{Y}\rightarrow \mathbb{R}$.
	
	Choose a sequence $\{r_k\}_k\subset (0,1)$ such that $r_k\rightarrow 0$ and then fix $k$. By Zorn's Lemma and the uniformly locally doubling property of $\mathscr{X}_\infty$, we can find a finite set $\{z_i: i \in I_k\} \subset B_{{R_\infty}+6}({x}_\infty)\cap \mathrm{X}_\infty$ such that
	$$
	B_{{R_\infty}+6}({x}_\infty) \subset \bigcup_{i \in I_k} B_{r_k}(z_i)\text{ and }\mathrm{d}(z_i, z_j) \geq r_k,\ \text{ for every $i, j\in I_k$ with $i \neq j$}.
	$$
	The fact that $\mathscr{X}_n\rightarrow \mathscr{X}_\infty$ implies that for every $i$, we can find a sequence $n\mapsto z_{n,i}\in B_{R_n+6}(x_n)\cap \mathrm{X}_n$ converging to $z_i$ in $(\mathrm{Y},\mathrm{d})$. Note that we can require that $\mathrm{d}(z_{n,i},z_i)\leq r_k/4$ for all $n$ and $i$. Hence, we have that for each $n$,
	$$
	B_{R_n+4}(x_n) \subset \bigcup_{i \in I_k} B_{5r_k/4}(z_{n,i})\text{ and }\mathrm{d}(z_{n,i}, z_{n,j}) \geq r_k/2,\ \text{ for every $i, j\in I_k$ with $i \neq j$},
	$$
	where we used the fact that $\mathrm{d}({x}_n,{x}_\infty)\in(0,1)$ for every $n$. Set $E_{n,i}=B_{5r_k/4}(z_{n,i})$, $E_{n,i}^{\prime}=B_{ 5(2\lambda+1) r_k/4}(z_{n,i})$ and $E_{i}=B_{5r_k/4}(z_{i})$. Now, we claim that for every $n$, the dilated balls $\{E_{n,i}^{\prime}\}_{i\in I_k}$ have a bounded overlap $L_1:=\mathrm{D}(5\lambda+11/4)^{2+\log_2(20\lambda+11)}$ in $\mathrm{X}_n$. Indeed, let $x\in \mathrm{X}_n$ and let $A_x=\{i\in I_k: x\in E_{n,i}^{\prime}\}$. Then
	\begin{align*}
		\mathfrak{m}_n(B_{[1+5(2\lambda+1)]r_k/4}(x))&\geq \sum_{i\in A_x}\mathfrak{m}_n(B_{r_k/4}(z_{n,i})) \\
		&\geq \mathrm{D}(5\lambda+11/4)^{-\log_2(80\lambda+44)}\mathfrak{m}_n(B_{[1+10(2\lambda+1)]r_k/4}(z_{n,i}))\\
		&\geq\mathrm{D}(5\lambda+11/4)^{-\log_2(80\lambda+44)}\mathfrak{m}_n(B_{[1+5(2\lambda+1)]r_k/4}(x))\sum_{i\in A_x} 1.
	\end{align*}
	Similarly, we can prove that the dilated balls $\{E_{i}\}_{i\in I_k}$ have a bounded overlap $L_2:=\mathrm{D}(7/4)^{1+\log_2 7}$ in $\mathrm{X}_\infty$. 
	
	By the weak local $p$-$p$ Poincar\'{e} inequality, we have that
	\begin{align}
		\sum_{i\in I_k} \int_{E_{n,i}}|f_n-(f_n)_{E_{n,i}}|^p \mathrm{d} \mathfrak{m}_n &\leq \frac{(5C)^p}{4^p}r_k^p L_1\int|D f_n|^p \mathrm{d} \mathfrak{m}_n,\notag\\
		\sum_{i\in I_k} \int_{E_{n,i}}|f_{n,m}-(f_{n,m})_{E_{n,i}}|^p \mathrm{d} \mathfrak{m}_n &\leq \frac{(5C)^p}{4^p}r_k^p L_1\int\mathrm{lip}_a^p(f_{n,m}) \mathrm{d} \mathfrak{m}_n\label{alrigeq4}\\
		&=\frac{(5C)^p}{4^p}r_k^p L_1\int\mathrm{lip}_a^p(f_{\infty,m}) \mathrm{d} \mathfrak{m}_n,\notag
	\end{align}
	where $C>0$ is independent of $n,k,m$. On the other hand, for all $n,m\geq 1$, by the H\"{o}lder inequality and the fact that $f_n=f_{n,m}=0$ on $\mathrm{X}_n\backslash B_{R_n+4}(x_n)\subset \mathrm{X}_n\backslash \cup_{i\in I_k}E_{n,i}$, we have that 
	\begin{align}
		\int|f_n-f_{n,m}|^p \mathrm{d} \mathfrak{m}_n &\leq 3^{p-1}  \sum_{i\in I_k}\bigg(\int_{E_{n,i}}|f_n-(f_n)_{E_{n,i}}|^p \mathrm{d} \mathfrak{m}_n \notag\\
		&\quad +\int_{E_{n,i}}|(f_{n,m})_{E_{n,i}}-f_{n,m}|^p \mathrm{d} \mathfrak{m}_n\label{alrigeq5}\\
		&\quad+\int_{E_{n,i}}|(f_n)_{E_{n,i}}-(f_{n,m})_{E_{n,i}}|^p \mathrm{d} \mathfrak{m}_n\bigg).\notag
	\end{align}
	By combining $(\ref{alrigeq4})$ and $(\ref{alrigeq5})$, and then letting $n\rightarrow \infty$, we obtain that
	\begin{align*}
		3^{1-p}\limsup_{n\rightarrow \infty}\int|f_n-f_{n,m}|^p \mathrm{d} \mathfrak{m}_n &\leq \frac{(5C)^p}{4^p}r_k^p  L_1\sup_n\int|D f_n|^p \mathrm{d} \mathfrak{m}_n \\
		&\quad +\frac{(5C)^p}{4^p}r_k^p  L_1\int\mathrm{lip}_a^p(f_{\infty,m}) \mathrm{d} \mathfrak{m}_\infty\\
		&\quad+ L_2\int |f_\infty-f_{\infty,m}|^p\mathrm{d} \mathfrak{m}_\infty.
	\end{align*}
	where we used the upper semi-continuity of $y\mapsto \mathrm{lip}_a(f_{\infty,m})(y)$. Finally by letting $m\rightarrow \infty$ and then letting $k\rightarrow \infty$, we get that
	\begin{align*}
		3^{1-p}\limsup_{m\rightarrow \infty}\limsup_{n\rightarrow \infty}\int|f_n-f_{n,m}|^p \mathrm{d} \mathfrak{m}_n =0.
	\end{align*}
	This means that $\{f_n\}_n$ $L^p$-strongly converges to $f_\infty$ by Proposition $\ref{alrigprop3}$.
\end{proof}

\subsection{Almost rigidity theorem}\label{alrig3}
In this and the following subsections, the following assumptions will be made:
\begin{assumption}\label{alrigassu}
	Let $p\in (1,\infty)$.
	\begin{itemize}
		\item Given $N\in (1,\infty)$, the sequence $\{\mathscr{X}_n:=[\mathrm{X}_n, \mathrm{d}_n, \mathfrak{m}_n,\bar{x}_n]\}_{n\in \bar{\mathbb{N}}_+}\subset \mathscr{M}_{\mathrm{D}_{0,N}(\cdot)}$ is such that $$\mathrm{d}_{\mathrm{pmGH}}(\mathscr{X}_n,\mathscr{X}_\infty)\rightarrow 0.$$  
		\item $\mathscr{X}_n$ is a $\mathrm{RCD}(0,N)$ space for every $n\in \bar{\mathbb{N}}_+$. 
		\item There is a proper metric space $(\mathrm{Y},\mathrm{d})$ containing isometrically all the $(\mathrm{X}_n, \mathrm{d}_n)$, $n\in \bar{\mathbb{N}}_+$.
		\item The sequence $\{\mathfrak{m}_n\}_{n}$ weakly converges to $\mathfrak{m}_\infty$ and $\mathrm{d}(\bar{x}_n,\bar{x}_\infty)\rightarrow 0$.
		\item Let $x_n\in \mathrm{X}_n$, $n\in \bar{\mathbb{N}}_+$, be such that $\mathrm{d}({x}_n,{x}_\infty)\in(0,1)$ for every $n$ and $\mathrm{d}({x}_n,{x}_\infty)\rightarrow 0$.
		\item Let $\{R_n\}_{n\in \bar{\mathbb{N}}_+}\subset (0,\infty)$ be such that $|R_n-R_\infty|\in(0,1)$ for every $n$ and $R_n\rightarrow R_\infty$.
	\end{itemize}
\end{assumption} 

The following proposition, which is crucial for our discussion, is excerpted from \cite[Theorem 8.1]{AH17}.
\begin{proposition}[$\Gamma$-convergence of $p$-Cheeger energies]\label{alrigprop8}
	$(\mathrm{a})$ For every sequence $n \mapsto f_n \in L^p(\mathrm{Y}, \mathfrak{m}_n)$ $L^p$-strongly converging to some $f_\infty \in L^p(\mathrm{Y}, \mathfrak{m}_\infty)$, it holds that
	$$
	\mathrm{Ch}_p(f_\infty) \leq \liminf _{n \rightarrow \infty} \mathrm{Ch}_p^n(f_n).
	$$
	$(\mathrm{b})$ For every $f_\infty \in L^p(\mathrm{Y}, \mathfrak{m}_\infty)$, there exists a sequence $n \mapsto f_n \in L^p(\mathrm{Y}, \mathfrak{m}_n)$ $L^p$-strongly convergent to $f_\infty$ such that
	$$
	\mathrm{Ch}_p(f_\infty) \geq \limsup _{n \rightarrow \infty} \mathrm{Ch}_p^n(f_n).
	$$
\end{proposition}

\begin{proposition}\label{alrigprop9}
	$(\mathrm{i})$ If the sequence $n\mapsto f_n\in W_0^{1,p}(B_{R_n}(x_n),\mathfrak{m}_n)$ $L^p$-weakly converges to $f_\infty\in L^p(\mathrm{Y}, \mathfrak{m}_\infty)$ and $\sup_n \|f_n\|_{W^{1,p}(\mathrm{Y},\mathfrak{m}_n)}<\infty$, then $\{f_n\}_n$ $L^p$-strongly converges to $f_\infty\in \hat{W}^{1,p}(B_{R_\infty}(x_\infty),\mathfrak{m}_\infty)$ and
	\begin{equation}\label{alrigeq9}\liminf_{n \rightarrow \infty} \int_{B_{R_n}(x_n)}|\nabla f_n|^p \mathrm{d}\mathfrak{m}_n \geq \int_{B_{R_\infty}(x_\infty)}|\nabla f_\infty|^p \mathrm{d}\mathfrak{m}_\infty.
	\end{equation}
	$(\mathrm{ii})$ For any $f_\infty \in W_0^{1, p}(B_{R_\infty}(x_\infty), \mathfrak{m}_\infty)$, there exists a sequence $$n\mapsto f_n \in W_0^{1, p}(B_{R_n}(x_n), \mathfrak{m}_n)$$ such that $\{f_n\}_n$ $W^{1, p}$-strongly converges to $f_\infty$.\\
	$(\mathrm{iii})$ If the sequence $n\mapsto f_n\in W_0^{1,p}(B_{R_n}(x_n),\mathfrak{m}_n)$ $W^{1, p}$-weakly converges to $f_\infty \in W^{1, p}(\mathrm{Y}, \mathfrak{m}_\infty)$, then
	\begin{equation}\label{alrigeq8}
		\liminf_{n \rightarrow \infty} \int g|\nabla f_n| \mathrm{d}\mathfrak{m}_n \geq \int g|\nabla f_\infty| \mathrm{d}\mathfrak{m}_\infty
	\end{equation}
	for any lower semi-continuous $g: \mathrm{Y} \rightarrow[0, \infty]$.\\
	$(\mathrm{iv})$ If the sequence $n\mapsto f_n\in W_0^{1,p}(B_{R_n}(x_n),\mathfrak{m}_n)$ $W^{1, p}$-strongly converges to $f_\infty$, then $\{|\nabla f_n|\}$ $L^p$-strongly converges to $|\nabla f_\infty|$.\\
	$(\mathrm{v})$ Let $g_n\in L^p(B_{R_n}(x_n),\mathfrak{m}_n)$, $n\in \bar{\mathbb{N}}_+$. If $\{g_n\}$ is $L^p$-strongly converging to $g_\infty$, then $\{\mu_{g_n}(t)\}_n$ converges to $\mu_{g_\infty}(t)$ for every $t\in (0,\infty)$, where $C$ is a countable set.\\
	$(\mathrm{vi})$Let $g_n \in W_0^{1, p}(B_{R_n}(x_n), \mathfrak{m}_n)$, $n\in \bar{\mathbb{N}}_+$. If $\{g_n\}$ $W^{1, p}$-weakly converges to $g_\infty$, then, up to a not relabeled subsequence, $\{g_n^{\star}\}$ converges to $g_\infty^{\star}$ in $L^p([0,\infty), \mathfrak{m}_{N})$.
\end{proposition}

\begin{proof}
	For part (i), since $n\mapsto f_n\in W_0^{1,p}(B_{R_n}(x_n),\mathfrak{m}_n)$ $L^p$-weakly converges to $f_\infty\in L^p(\mathrm{Y}, \mathfrak{m}_\infty)$ and $\sup_n \|f_n\|_{W^{1,p}(\mathrm{Y},\mathfrak{m}_n)}<\infty$, by Theorem $\ref{compact}$ we have that $\{f_n\}_n$ $L^p$-strongly converges to $f_\infty\in \hat{W}^{1,p}(B_{R_\infty}(x_\infty),\mathfrak{m}_\infty)$. Then $(\ref{alrigeq9})$ follows from Proposition $\ref{alrigprop8}$.
	
	For parts (ii), (iii), (iv), (v) and (vi), we refer to \cite{MV21} and \cite{Wu24}. 
\end{proof}

We can now prove the main result almost rigidity result of this note.
\begin{theorem}[Almost rigidity theorem]\label{alrig}
	For every $\epsilon>0$, $p>1$, $N>1$, $0<a_1<a_2<\infty$, $b,d \in(0,\infty)$, $0<c_l \leq c_u<\infty$, there exists $\delta=\delta(\epsilon, p, N, a_1,a_2,b, c_l / c_u,d)>0$ such that the following statement holds.
	Assume that
	\begin{itemize}
		\item $(\mathrm{X}, \mathrm{d}, \mathfrak{m},\bar{x})$ is a $\operatorname{RCD}(0, N)$ space with $\mathrm{AVR}_{\mathfrak{m}}\geq b$ and $\Omega=B_R(x) \subset \mathrm{X}$ is an open ball with $\mathfrak{m}(\Omega):=a\in [a_1,a_2]$ and $\mathrm{d}(\bar{x},x)\leq d;$
		\item $f \in W_0^{1, q}(\Omega, \mathfrak{m})$ with $c_l \leq\|f\|_{L^q(\Omega, \mathfrak{m})} \leq\|f\|_{W^{1, q}(\Omega, \mathfrak{m})} \leq c_u;$
		\item $u \in W_0^{1, p}(\Omega, \mathfrak{m})$ is a solution of $-\mathscr{L}_p u=f;$
		\item $v \in W_0^{1, p}([0, r_a), \mathfrak{m}_{N})$ is a solution of $-\mathscr{L}_p v=f^{\star}$, where $r_a=H_N^{-1}(a)$.
	\end{itemize}
	If $\|\mathrm{AVR}_{\mathfrak{m}}^{q/N} u^{\star}-v\|_{L^p([0, r_a), \mathfrak{m}_{N})}<\delta$, then there exists a Euclidean metric measure cone $(\mathrm{Z}, \mathrm{d}_{\mathrm{Z}}, \mathfrak{m}_{\mathrm{Z}},\bar{z})$ such that
	$$
	\mathrm{d}_{\mathrm{pmGH}}((\mathrm{X}, \mathrm{d}, \mathfrak{m},\bar{x}),(\mathrm{Z}, \mathrm{d}_{\mathrm{Z}}, \mathfrak{m}_{\mathrm{Z}},\bar{z}))<\epsilon .
	$$
\end{theorem}
\begin{proof}
	We argue by contradiction.
	
	Step 1: Assume that there exist ${\epsilon},p, {N}, {a}_1,{a}_2,{b},{c},d$ such that for every $n \in \mathbb{N}_{+}$, we can find a $\operatorname{RCD}(0, {N})$ space  $(\mathrm{X}_n, \mathrm{d}_n, \mathfrak{m}_n,\bar{x}_n)$, a ball $\Omega_n=B_{R_n}(x_n) \subset \mathrm{X}_n$ with $\mathfrak{m}_n(\Omega_n):=a_n\in [{a}_1,{a}_2]$ and $\mathrm{d}(\bar{x}_n,x_n)\leq d$, and functions $f_n \in W_0^{1, q}(\Omega_n, \mathfrak{m}_n)$, $u_n \in W_0^{1, p}(\Omega_n, \mathfrak{m}_n)$, $v_n \in W_0^{1, p}([0, r_{{a}_n})$, $\mathfrak{m}_{{N}})$ such that 
	\begin{align*}
		& -\mathscr{L}_p u_n=f_n \text { weakly, }-\mathscr{L}_p v_n=f_n^{\star} \text { weakly, }\quad  \mathrm{AVR}_{\mathfrak{m}_n}\geq b,\\
		& {c} \leq\|f_n\|_{L^q(\Omega_n, \mathfrak{m}_n)} \leq\|f_n\|_{W^{1, q}(\Omega_n, \mathfrak{m}_n)} \leq 1, \quad\|\mathrm{AVR}_{\mathfrak{m}_n}^{q/N}u_n^{\star}-v_n\|_{L^p([0, r_{a_n}), \mathfrak{m}_{{N}})}<\frac{1}{n},
	\end{align*}
	and moreover
	\begin{align*}
		\inf \big\{\mathrm{d}_{\mathrm{pmGH}}((\mathrm{X}_n, \mathrm{d}_n, \mathfrak{m}_n,\bar{x}_n),(\mathrm{Z}, \mathrm{d}_{\mathrm{Z}}, \mathfrak{m}_{\mathrm{Z}},\bar{z}))&:(\mathrm{Z}, \mathrm{d}_{\mathrm{Z}}, \mathfrak{m}_{\mathrm{Z}},\bar{z}) \\
		&\text { is a Euclidean metric measure cone}\big\} \geq {\epsilon} .
	\end{align*}
	We claim that $\{R_n\}_n$ is bounded from above and bounded away from $0$. Indeed, we have that
	\begin{align*}
		b\omega_N R_n^N&\leq \mathrm{AVR}_{\mathfrak{m}_n}\omega_N R_n^N\leq \mathfrak{m}_n(B_{R_n}(x_n))\leq a_2,\\
		R_n^N&\geq \frac{\mathfrak{m}_n(B_{R_n}(x_n))}{\mathfrak{m}_n(B_{1}(x_n))}\geq \frac{a_1}{\sup_n\mathfrak{m}_n(B_{1+d}(\bar{x}_n))}>0.
	\end{align*}
	This implies the validity of the above claim.

	Now, up to not relabeled subsequences, we can assume that $\{(\mathrm{X}_n, \mathrm{d}_n, \mathfrak{m}_n,\bar{x}_n)\}$ converges to a $\operatorname{RCD}(0,{N})$ space $(\mathrm{X}_\infty, \mathrm{d}_\infty, \mathfrak{m}_\infty,\bar{x}_\infty)$ (by Proposition $\ref{alrigprop7}$), Assumption $\ref{alrigassu}$ is satisfied, and $\mathrm{AVR}_{\mathfrak{m}_n}\rightarrow A\in [b,\mathrm{AVR}_{\mathfrak{m}_\infty}]$ (because $\limsup_{n\rightarrow \infty}\mathrm{AVR}_{\mathfrak{m}_n}\leq \mathrm{AVR}_{\mathfrak{m}_\infty}<\infty$). Hence, we have that $a_\infty:=\mathfrak{m}_\infty (B_{R_\infty}(x_\infty))\in [{a}_1,{a}_2]$.
	
	Step 2: We claim that, up to not relabeled subsequences, $\{f_n\}$ $L^q$-strongly converges to a non-zero function $f_\infty \in L^q(B_{R_\infty}(x_\infty), \mathfrak{m}_\infty)$ and $\{f_n^{\star}\}$ converges to $f_\infty^{\star}$ in $L^q([0, \infty), \mathfrak{m}_{{N}})$, $\{u_n\}$ $L^p$-strongly converges to a function $u_\infty \in \hat{W}_0^{1, p}(B_{R_\infty}(x), \mathfrak{m})$ and $\{u_n^{\star}\}$ converges to $u^{\star}$ in $L^p([0, \infty, \mathfrak{m}_{{N}})$, and $\{v_n\}$ converges to a function $v_\infty \in W_0^{1, p}([0, r_{a_n}), \mathfrak{m}_{ {N}})$ in $L^p([0,\infty), \mathfrak{m}_{{N}})$. Indeed, first we note that the following identity was proved in $(\ref{ex-so-2})$:
	\begin{equation}\label{alrigeq12}
		v_n(\rho)=\int_{H_{N}(\rho)}^{H_{ N}(r_{a_n})} \frac{1}{\mathcal{I}_{ {N}}(\sigma)}\bigg(\frac{1}{\mathcal{I}_{{N}}(\sigma)} \int_0^\sigma f_n^\star \circ H_{ {N}}^{-1}(t) \mathrm{d} t\bigg)^{\frac{1}{p-1}} \mathrm{d} \sigma, \quad \text{ for any } \rho \in[0, r_{a_n}].
	\end{equation}
	From the fact that $\sup_n\|f_n\|_{W^{1, q}} \leq 1$ by assumption, it follows that $\sup _n\|f_n^{\star}\|_{L^q} \leq 1$ by Proposition $\ref{propre}$(b). The estimates $(\ref{2.21})$, $(\ref{2.22})$ and $(\ref{2.23})$ imply that $\{v_n\}$ is also bounded in $W^{1, p}([0, r_{{a}_n}), \mathfrak{m}_{{N}})$. Using the Talenti-type comparison theorem and Proposition $\ref{propre}$(b) again, we deduce that $\sup_n\|u_n\|_{W^{1, p}}<\infty$. Thus, by Proposition $\ref{alrigprop9}$(i), (and up to not relabeled subsequences), $\{u_n\}$ $L^p$-strongly converges to $u_\infty \in \hat{W}_0^{1, p}(B_{R_\infty}(x_\infty), \mathfrak{m}_\infty)$, $\{f_n\}$ $L^q$-strongly converges to $f_\infty \in \hat{W}_0^{1, q}(B_{R_\infty}(x_\infty), \mathfrak{m}_\infty)$ and $\{v_n\}$ $L^p$-strongly converges to $v_\infty \in W_0^{1, p}([0, r_{a_\infty}), \mathfrak{m}_{{N}})$. In particular, $\{v_n\}$ converges to $v_\infty$ in $L^p([0, \infty), \mathfrak{m}_{{N}})$ by uniform convexity of $L^p$ spaces. Moreover, by Proposition $\ref{alrigprop9}$(vi), $\{u_n^{\star}\}$ converges to $u_\infty^{\star}$ in $L^p([0, \infty), \mathfrak{m}_{{N}})$ and $\{f_n^{\star}\}$ converges to $f_\infty^{\star}$ in $L^q([0, \infty), \mathfrak{m}_{{N}})$.
	
	Step 3: We claim that $v_\infty$ is a solution to $-\mathscr{L}_p v=f_\infty^{\star}$. The facts that $a_n\rightarrow a_\infty$, $\{v_n\}$ converges to $v_\infty$ in $L^p([0, \infty), \mathfrak{m}_{{N}})$, $\{f_n^{\star}\}$ converges to $f^{\star}$ in $L^q([0, \infty), \mathfrak{m}_{{N}})$ and $(\ref{alrigeq12})$ imply that $v_\infty$ can be represented as
	\begin{equation}\label{alrigeq11}
		v_\infty(\rho)=\int_{H_{{N}}(\rho)}^{H_{{N}}(r_{a_\infty})} \frac{1}{\mathcal{I}_{ {N}}(\sigma)}\bigg(\frac{1}{\mathcal{I}_{{N}}(\sigma)} \int_0^\sigma f_\infty^{\star} \circ H_{{N}}^{-1}(t) \mathrm{d} t\bigg)^{\frac{1}{p-1}} \mathrm{d} \sigma, \quad \text{ for any } \rho \in[0, r_{a_\infty}].
	\end{equation}
	We obtain the claim by Proposition $\ref{propex}$.
	
	Step 4. Notice that
	\begin{equation}\label{alrigeq10}
		\mathrm{d}_{\mathrm{pmGH}}((\mathrm{X}_\infty, \mathrm{d}_\infty, \mathfrak{m}_\infty,\bar{x}_\infty),(\mathrm{Z}, \mathrm{d}_{\mathrm{Z}}, \mathfrak{m}_{\mathrm{Z}},\bar{z})) \geq {\epsilon}
	\end{equation}
	for any Euclidean metric measure cone $(\mathrm{Z}, \mathrm{d}_{\mathrm{Z}}, \mathfrak{m}_{\mathrm{Z}},\bar{z})$. However, by $\mathrm{AVR}_{\mathfrak{m}_n}\rightarrow A$, the $L^p$-strong convergence of $u_n^{\star}$ to $u_\infty^{\star}$ and the $L^p$-strong convergence of $v_n$ to $v_\infty$, one has
	$$
	\|A^{q/{N}}u_\infty^{\star}-v_\infty\|_{L^p}=\lim _{n \rightarrow \infty}\|\mathrm{AVR}_{\mathfrak{m}_n}^{q/{N}}u_n^{\star}-v_n\|_{L^p}=0
	$$
	which implies that $A^{q/{N}}u_\infty^{\star}=v_\infty$ and thus $A^{q/{N}}u_\infty^{\sharp}=v_\infty^\sharp$. This implies that $\mu_{u_\infty}(A^{-q/{N}}t)=\nu(t)$ for any $t>0$, where $\nu$ is the distribution function of $v_\infty$. Moreover, since $f_\infty$ is the $L^q$-strong limit of the sequence $\{f_n\}$, it has $L^q$-norm bounded from below by ${c}$ and thus $v_\infty \neq 0$. By the Talenti-type comparison theorem and $A\leq \mathrm{AVR}_{\mathfrak{m}_\infty}$, we can deduce that
	$$
	\mathrm{AVR}_{\mathfrak{m}_\infty}^{q/{N}}u_\infty^{\star}\geq A^{q/{N}}u_\infty^{\star}=v_\infty \geq \mathrm{AVR}_{\mathfrak{m}_\infty}^{q/{N}}u_\infty^{\star}.
	$$
	This implies that $A=\mathrm{AVR}_{\mathfrak{m}_\infty}$.

	The fact that each $u_n$ is a solution to $(\ref{eq1})$ implies that for $t, h>0$, it holds
	\begin{align*}
		\frac{1}{h} \int_{\{t<|u_n| < t+h\}}|\nabla u_n| \mathrm{d}\mathfrak{m}_n &\leq  \bigg(-\frac{\mu_{u_n}(t+h)-\mu_{u_n}(t)}{h}\bigg)^{1 / q} \frac{1}{h^{\frac{1}{p}}}\bigg(\int(u_n-t)^+ f_n \mathrm{d}\mathfrak{m}_n \\
		& \quad-\int(-u_n-t)^+ f_n \mathrm{d}\mathfrak{m}_n-\int(u_n-t-h)^+ f_n \mathrm{d}\mathfrak{m}_n \\
		&\quad+\int(-u_n-t-h)^+ f_n \mathrm{d}\mathfrak{m}_n\bigg),
	\end{align*}
	see \cite[Lemma 3.2]{Wu24} for the details. By letting $n \rightarrow \infty$ and then letting $h \rightarrow 0$, we obtain that
	$$
	-\frac{\mathrm{d}}{\mathrm{d} t} \int_{\{|u_\infty|>t\}}|\nabla u_\infty| \mathrm{d}\mathfrak{m}_\infty \leq(-\mu_{u_\infty}^{\prime}(t))^{1 / q}\bigg(\int_{\{|u_\infty|>t\}}|f_\infty| \mathrm{d}\mathfrak{m}_\infty\bigg)^{1 / p}, \quad  \mathcal{L}^1 \text {-a.e. } t>0,
	$$
	where we used Lemma $\ref{Talem0}$ and Propositions $\ref{alrigprop2}$, $\ref{2.18}$(i), $\ref{alrigprop9}$(iii), $\ref{alrigprop9}$(v). Combining this, Theorem $\ref{thmiso}$, coarea formula $(\ref{coeq4})$, Chain rule of minimal $p$-relaxed slope, Proposition $\ref{propre}$(c) and $(\ref{alrigeq11})$, we obtain that
	\begin{align*}
		1 &\leq \frac{-\mu_{u_\infty}^{\prime}(t)}{\mathrm{AVR}_{\mathfrak{m}_\infty}^{\frac{q}{{N}}}\mathcal{I}_{{N}}(\mu_{u_\infty}(t))^{\frac{p}{p-1}}}\bigg(\int_0^{\mu_{u_\infty}(t)} f_\infty^{\sharp}(s) \mathrm{d} s\bigg)^{\frac{1}{p-1}}\\
		&=\nu^{\prime}(\mathrm{AVR}_{\mathfrak{m}_\infty}^{\frac{q}{{N}}}t) (v_\infty^\sharp)^{\prime}(\nu(\mathrm{AVR}_{\mathfrak{m}_\infty}^{\frac{q}{{N}}}t))=1,
	\end{align*}
	for a.e. $t \in(0, \operatorname{ess-sup}|u_\infty|)$. Hence, for a.e. $t \in(0, \operatorname{ess-sup}|u_\infty|)$, the superlevel set $\{|u_\infty|>t\}$ satisfies $\mathrm{AVR}_{\mathfrak{m}_\infty}^{1/{N}}\mathcal{I}_{{N}}(\mathfrak{m}_\infty(\{|u_\infty|>t\}))=\operatorname{Per}(\{|u_\infty|>t\})$. By the rigidity in the isoperimetric inequality, this implies that $(\mathrm{X}_\infty, \mathrm{d}_\infty, \mathfrak{m}_\infty,\bar{x}_\infty)$ is a Euclidean metric measure cone, contradicting $(\ref{alrigeq10})$.
\end{proof}

\subsection{The continuity of the Poisson problem}\label{alrig4}
The proof of the following result is very similar to the proof in \cite[Theorem 3.19]{Wu24}, and it omitted.
\begin{theorem}\label{conPoi}
	Under the assumption $\ref{alrigassu}$, let $a_n:=\mathfrak{m}_n(B_{R_n}(x_n))$ and let $f_n \in W_0^{1, q}(B_{R_n}(x_n), \mathfrak{m}_n)$ with $\sup _n\|f_n\|_{W^{1, q}}<\infty$. Assume that
	\begin{equation}\label{alrigeq13}
		W_0^{1, p}( B_{R_\infty}(x_\infty), \mathfrak{m}_\infty)=\hat{W}_0^{1, p}( B_{R_\infty}(x_\infty), \mathfrak{m}_\infty) ,
	\end{equation}
	and for each $n$, $u_n \in W_0^{1, p}(B_{R_n}(x_n), \mathfrak{m}_n)$ is a solution to
	$$
	\begin{cases}-\mathscr{L}_p u=f, & \text { in } B_{R_n}(x_n), \\ u=0, & \text { on } \partial B_{R_n}(x_n),\end{cases}
	$$
	and $v_n \in W_0^{1, p}([0, r_{a_n}), \mathfrak{m}_{N})$ is the solution to
	$$
	\left\{\begin{array}{l}
		-\mathscr{L}_p v=f_n^{\star} \text { in }[0, r_{a_n}) \\
		v(r_{a_n})=0.
	\end{array}\right.
	$$
	Then, up to subsequences, we have that \\
	$(\mathrm{i})$ $\{f_n\}$ $L^q$-strongly converges to a function $f_\infty \in L^q(B_{R_\infty}(x_\infty), \mathfrak{m}_\infty)$ and $\{f_n^{\star}\}$ converges to $f_\infty^{\star}$ in $L^q([0, r_{a_\infty}), \mathfrak{m}_{N})$, where $a_\infty:=\mathfrak{m}_\infty(B_{R_\infty}(x_\infty))$.\\
	$(\mathrm{ii})$ $\{v_n\}$ $W^{1, p}$-strongly converges to a solution $v_\infty$ to
	$$
	\left\{\begin{array}{l}
		-\mathscr{L}_p v=f_\infty^{\star} \text { in }[0, r_{a_\infty}) \\
		v(r_{a_\infty})=0
	\end{array}\right.
	$$
	If $p \geq 2$, in addition, then, up to a subsequence,\\
	$(\mathrm{iii})$ $\{u_n\}$ $W^{1, p}$-strongly converges to a solution $u_\infty$ to
	$$
	\begin{cases}-\mathscr{L}_p u=f & \text { in } B_{R_\infty}(x_\infty) \\ u=0 & \text { on } \partial B_{R_\infty}(x_\infty).\end{cases}
	$$
\end{theorem}

\begin{remark}
	We remark that the condition $(\ref{alrigeq13})$ is necessary. We consider a trivial $\mathrm{pmGH}$ convergent sequence and a convergent sequence of points:
	$$
	([0,\infty), \mathrm{d}_{eu}, \mathfrak{m}_{N},0) \xrightarrow{\mathrm{pmGH}}([0,\infty), \mathrm{d}_{eu}, \mathfrak{m}_{N},0) , \text { and } x_n:=1-\epsilon_n \rightarrow x:=1
	$$
	as $\epsilon_n \downarrow 0$. For simplicity, we set $s_n=2-\epsilon_n$. Let $0 \neq f \in \operatorname{Lip}_c((0, 1))$ be a nonnegative function. Let $R_n={R_\infty}=1$ such that $B_{R_n}(x_n)=[0, s_n)$ and $B_{R_\infty}(x)=(0, 2)$, and let $f_n=f$. Proposition $\ref{propex}$ says that the unique solution of the Poisson problem
	$$
	\left\{\begin{array}{l}
		\mathscr{L}_pv=f_n, \text { in }[0, s_n), \\
		v(s_n)=0,
	\end{array}\right.
	$$
	can be represented as
	$$
	v_n(\rho)=\int_{H_{N}(\rho)}^{H_{N}(s_n)} \frac{1}{\mathcal{I}_{N}(\sigma)}\left(\frac{1}{\mathcal{I}_{N}(\sigma)} \int_0^\sigma f \circ H_{N}^{-1}(t) \mathrm{d} t\right)^{\frac{1}{p-1}} \mathrm{d} \sigma, \quad \text { for any } \rho \in[0, s_n].
	$$
	The $W^{1, p}$ strong limit function of $\{v_n\}$
	$$
	v_\infty(\rho)=\int_{H_{N}(\rho)}^{H_{N}(2)} \frac{1}{\mathcal{I}_{ N}(\sigma)}\left(\frac{1}{\mathcal{I}_{ N}(\sigma)} \int_0^\sigma f \circ H_{ N}^{-1}(t) \mathrm{d} t\right)^{\frac{1}{p-1}} \mathrm{d} \sigma ,\quad \text { for any } \rho \in[0, 2],
	$$
	is continuous on $[0, 2]$ and $v(0) \neq 0$, so it is not in the space $W_0^{1, p}((0, 2), \mathfrak{m}_{N})$ if $p>N$.
\end{remark}

\begin{ack}
The author would like to express her gratitude to Professor Jiayu Li for his encouragement and to Professor Huichun Zhang for his discussions. 
\end{ack}
\end{document}